\documentclass[a4paper]{article}

\usepackage[latin1]{inputenc}
\usepackage[french,british]{babel}
\usepackage[T1]{fontenc}
\usepackage{setspace}
\usepackage{indentfirst} 
\usepackage{fancyhdr} 
\usepackage{nicefrac}
\usepackage{textcomp}
\usepackage{lastpage}
\usepackage{amssymb,amsfonts,amsmath,amsthm}
\usepackage{latexsym}
\usepackage{url}
\usepackage{verbatim}
\usepackage[all]{xy}
\usepackage{mathrsfs}
\usepackage{stmaryrd}
\usepackage{oldgerm}
\usepackage{bm}
\usepackage{graphics}
\usepackage{wrapfig}
\usepackage{fancybox}
\usepackage{graphicx}
\input xy
\xyoption{all}

\theoremstyle{plain}
\newtheorem{theo}{Theorem}[subsection]
\newtheorem*{theo*}{Theorem}
\newtheorem{lem}[theo]{Lemma}
\newtheorem*{lem*}{Lemma}
\newtheorem{cor}[theo]{Corollary}
\newtheorem{prop}[theo]{Proposition}
\newtheorem*{prop*}{Proposition}

\newtheorem*{conj*}{Conjecture}

\theoremstyle{remark}
\newtheorem{rem}[theo]{Remark}

\theoremstyle{definition}
\newtheorem{defi}[theo]{Definition}

\newtheorem{ex}[theo]{Example}

\newtheorem*{Qu}{Question}

\numberwithin{equation}{subsection}

\DeclareMathOperator{\NN}{\mathbb{N}}
\DeclareMathOperator{\ZZ}{\mathbb{Z}}
\DeclareMathOperator{\QQ}{\mathbb{Q}}
\DeclareMathOperator{\RR}{\mathbb{R}}

\DeclareMathOperator{\PP}{\mathbb{P}}

\DeclareMathOperator{\HH}{\mathbb{H}}
\DeclareMathOperator{\aA}{\mathbb{A}}

\DeclareMathOperator{\Ker}{Ker}
\DeclareMathOperator{\Hom}{Hom}

\DeclareMathOperator{\Aut}{Aut}
\DeclareMathOperator{\id}{id}
\DeclareMathOperator{\OO}{\mathscr{O}}

\DeclareMathOperator{\GL}{\mathbf{GL}}

\DeclareMathOperator{\Spec}{Spec}
\DeclareMathOperator{\Spf}{Spf}
\DeclareMathOperator{\Proj}{Proj}

\DeclareMathOperator{\ord}{ord}

\DeclareMathOperator{\MM}{\mathscr{M}}
\DeclareMathOperator{\h}{H}

\DeclareMathOperator{\Pic}{Pic}

\DeclareMathOperator{\GD}{\mathscr{G}}

\DeclareMathOperator{\GG}{\mathbb{G}}

\DeclareMathOperator{\PU}{\mathscr{P}}

\DeclareMathOperator{\ver}{\mathscr{V}}
\DeclareMathOperator{\edg}{\mathscr{E}}

\DeclareMathOperator{\rank}{rank}

\DeclareMathOperator{\K}{K}

\DeclareMathOperator{\gr}{gr}

\DeclareMathOperator{\Tor}{Tor}

\DeclareMathOperator{\F}{\mathscr{F}}

\DeclareMathOperator{\Supp}{Supp}
\DeclareMathOperator{\f}{F}
\DeclareMathOperator{\V}{V}
\DeclareMathOperator{\CH}{CH}

\DeclareMathOperator{\Ksheaf}{\mathscr{K}}
\DeclareMathOperator{\yer}{\mathscr{Y}}

\DeclareMathOperator{\DF}{\mathbf{R}}
\DeclareMathOperator{\R}{R}
\DeclareMathOperator{\FaiAb}{\mathfrak{S}}
\DeclareMathOperator{\Fan}{\mathfrak{ST}}

\DeclareMathOperator{\Saut}{\mathscr{A}\mathit{ut}}

\DeclareMathOperator{\sgn}{sgn}

\begin{document}

\selectlanguage{british}

\thispagestyle{empty}

\title{Overconvergent Chern Classes and Higher Cycle Classes}
\author{Veronika Ertl\\
Universit\"at Regensburg}
\date{}
\maketitle

\begin{abstract}
{\noindent
The goal of this work is to construct integral Chern classes and higher cycle classes for a smooth variety over a perfect field of characteristic $p>0$ that are compatible with the rigid Chern classes defined by Petrequin. The Chern classes we define have coefficients in the overconvergent de Rham--Witt complex of Davis, Langer and Zink and the construction is based on the theory of cycle modules discussed by Rost. We prove a comparison theorem in the case of a quasiprojective variety.}
\end{abstract}
\vspace{-9pt}
\selectlanguage{french}
\begin{abstract}
{\noindent
Le but de ce travail est de construir des classes de Chern enti\`eres et des classes de cycles pour une vari\'et\'e lisse sur un corps parfait de caract\'eristique $p>0$ compatible aux classes de Chern rigides d\'efinies par Petrequin. Les classes de Chern que l'on d\'efinit sont \`a co\'efficientes dans le complexe de de Rham--Witt surconvergent de Davis, Langer et Zink et la construction repose sur la th\'eorie de modules de cycles discut\'ee par Rost. On d\'emontre un th\'eor\`eme de comparaison dans le cas d'une vari\'et\'e quasiprojective.
}
\end{abstract}

\selectlanguage{british}

\tableofcontents

\addcontentsline{toc}{section}{Introduction}
\section*{Introduction}

It is well known that crystalline cohomology is a good integral model for Berthelot's rigid cohomology in the case of a proper variety. The overconvergent de Rham-Witt cohomology introduced by Davis, Langer and Zink \cite{DavisLangerZink} is an integral $p$-adic cohomology theory for smooth varieties designed to be compatible with Monsky-Washnitzer cohomology in the affine case and with rigid cohomology in the quasi-projective case. 

In view of the fact that for proper smooth varieties over a field of characteristic $p>0$ crystalline Chern classes are integral Chern classes which are according to Petrequin compatible with the rigid ones \cite{Petrequin}, the following question is reasonable:

\begin{Qu} Can we define integral Chern classes for (open) smooth varieties that are compatible with the rigid ones?\end{Qu}

We use the above mentioned overconvergent de Rham-Witt complex as an obvious choice for coefficients for integral Chern classes on smooth varieties.

Let $X$ be a smooth variety over a perfect field $k$ of positive characteristic $p>0$. Denote by $W^\dagger\Omega_X$ the \'etale sheaf of overconvergent Witt differentials over $X$. We construct a theory of higher Chern classes with coefficients in the overconvergent complex
$$c_{ij}^{\text{sc}}:K_j(X)\rightarrow\HH^{2i-j}(X,W^\dagger\Omega_X).$$
If $X$ is quasi-projective we prove the following comparison:
\begin{prop*} The overconvergent Chern classes $c_{ij}^{\text{sc}}:K_j(X)\rightarrow\HH^{2i-j}(X,W^\dagger\Omega_X)$ are compatible with the rigid Chern classes $c_{ij}^{\text{rig}}:K_j(X)\rightarrow\h_{\text{rig}}^{2i-j}(X/K)$ defined by Petrequin \cite{Petrequin} via the comparison morphism of \cite{DavisLangerZink}.\end{prop*}

Or more explicitly, the following diagram commutes for all i,j:
\begin{equation}\label{DiagComp}\xymatrix{& & \HH^{2i-j}(X,W^\dagger\Omega_X)\otimes\QQ\\
K_j(X)\ar[rru]^{c_{ij}^{\text{sc}}}\ar[rrd]_{c_{ij}^{\text{rig}}} & &\\
& &\h_{\text{rig}}^{2i-j}(X/K)\ar[uu]^{\cong}}\end{equation}
where the vertical map is the comparison isomorphism.

Let us now present the different parts of the article.

We begin by recalling facts about Milnor $K$-theory for local rings, including the Gersten conjecture for the associated sheaf on a scheme $X$ where all residue fields have ``enough'' elements due to Kerz \cite{Kerz}. We note that the Gersten conjecture implies that the ``na\"ive'' definition of the Milnor $K$-sheaf as the sheaf associated to the pre-sheaf given for a ring $A$ by
$$\overline{K}_\ast^M(A)=T^\ast(A)\slash\text{ Steinberg relations }$$
coincides with the definition used by Rost \cite{Rost} denoted by $\Ksheaf_\ast^M$. 

In order to be able to apply our results to a more general case, we describe Kerz's improved Milnor $K$-theory for local rings with finite residue fields and note that his results hold for \'etale topology, too.

In Section \ref{Rost} we state Rost's axiomatic approach to Chow groups in terms of cycle modules. We will later use the fact that the Milnor $K$-ring is a cycle module. An important result is Corollary \ref{Cor65} which makes it possible to calculate the cohomology of a cycle module in terms of the associated Chow groups. In particular, we can calculate the cohomology of the Milnor $K$-sheaf in terms of the cohomology of the associated cycle complex. We make use of this in the proof of the Projective Bundle Formula (Proposition \ref{ProjectiveBundleFormula}), sketched by Gillet in his survey \cite{Gillet2}. The statement is proven for Chow groups in general, but is in particular applied to the Milnor $K$-sheaf in the next section in order to show that it gives rise a duality theory.

In Section \ref{Gillet} we mention of Gillet's generalised duality theories. In Theorem \ref{TheoChernClasses} and Theorem \ref{TheoHigherChernClasses} we recall his result that for a duality theory $\Gamma(\ast)$ satisfying such axioms there exists a theory of higher Chern classes
$$c_{ij}:K_j(X)\rightarrow\h^{di-j}(X,\Gamma(i)).$$
We now define a duality theory by setting
$$\underline{\Gamma}^\ast_X(j)=\Ksheaf_j^M.$$
As we show that it satisfies Gillet's axioms, we conclude in Theorem \ref{TheoKsheafChern} that there is a theory of Chern classes with coefficients in the Milnor $K$-sheaf
$$c_{ij}:K_j(X)\rightarrow \h^{i-j}(X,\Ksheaf_n^M).$$

Assume now that $k$ is a perfect field of characteristic $p>0$ and $X$ a smooth $k$-scheme. 

In Section \ref{Logarithmic} we recall the definition of the overconvergent de Rham--Witt complex $W^\dagger\Omega_X$ introduced by Davis, Langer and Zink in \cite{DavisLangerZink}. It is easy to see that logarithmic Witt differentials are in fact overconvergent. This leads us to define for every $i\geqslant 0$ a morphism
\begin{eqnarray*}d\log^i:\OO^\ast_X\otimes\cdots\OO^\ast_X &\rightarrow& W\Omega^i_{X,\log} \rightarrow W^\dagger\Omega_X\left[i\right]\\
x_1\otimes\cdots\otimes x_i &\mapsto& d\log(x_1)\cdots d\log(x_i).\end{eqnarray*}
In Proposition \ref{PropSteinberg} we prove that the symbols $d\log(x_1)\cdots d\log(x_i)$ satisfy the Steinberg relation. Therefore the morphism $d\log^i$ factors through the na\"ive Milnor $K$-sheaf
$$d\log^i:\overline{\Ksheaf}_i^M\rightarrow W^\dagger\Omega\left[i\right].$$
We show that the overconvergent de Rham--Witt complex has a transfer map or norm which satisfies the conditions given in \cite{Kerz3}. Moreover, it is continuous. As a consequence we obtain for each $i$ a unique natural transformation
$$\widehat{d\log^i}:\widehat{\Ksheaf}_i^M\rightarrow W^\dagger\Omega\left[i\right],$$
and we don't have to distinguish any more between the different definitions of Milnor $K$-theory.

This enables us by functoriality of sheaf cohomology to define in Section \ref{Overconvergent} Chern classes with coefficients in the overconvergent complex induced by the ones for Milnor $K$-theory and prove some basic facts about them. 
\begin{theo*} There is a theory of Chern classes for vector bundles and higher algebraic $K$-theory of regular varieties over $k$ with infinite residue fields, with values with coefficients in the overconvergent de Rham--Witt complex:
$$c_{ij}^{\text{sc}}:K_j(X)\rightarrow\HH^{2i-j}(X,W^\dagger\Omega_X).$$
\end{theo*}

In Section \ref{Sec2.2} we compare these Chern classes to crystalline Chern classes using a construction due to Gros \cite{Gros}. Section \ref{AppendixGamma} examines the behaviour of the map $c^{\text{sc}}_{ij}$ on the $\gamma$-graded pieces of algebraic $K$-theory. In Section \ref{SubsecShortExactSequence} we attempt to establish a short exact sequence 
$$0\rightarrow W^\dagger\Omega^i_{X,\log}\rightarrow W^\dagger\Omega^i_X\xrightarrow{\f-1} W^\dagger\Omega^i_X\rightarrow 0$$
used to obtain Chern classes into Frobenius fixed submodules of $W^\dagger\Omega$.

As a preparation for our comparison theorem we go in Section \ref{RigChernClasses} over Petrequin's definition of rigid Chern classes and how to calculate them with \v{C}ech cocycles. We show that they factor through Milnor $K$-theory. From now on we assume that $X/k$ is smooth and quasi-projective. In this case Davis, Langer and Zink construct a rigid-overconvergent comparison morphism, which we recall in Section \ref{CompDLZ}. In fact, they show that there is a natural quasi-isomorphism
$$R\Gamma_{\text{rig}}(X)\rightarrow R\Gamma(X,W^\dagger\Omega_{X/k})\otimes\QQ.$$
In the last part of Section \ref{Comparison}, we show that the overconvergent Chern classes that we constructed are compatible with Petrequin's rigid Chern classes via this comparison map. This relies on the fact that they both factor through Milnor $K$-theory and we have a commutative diagram
$$\xymatrix{&&&\h^{2j-i}_{\text{rig}}(X/K)\\
K_j(X)\ar[rrru]^{c_{ij}^{\text{rig}}}\ar[rr]|{c_{ij}^M}\ar[rrrd]_{c_{ij}^{\text{sc}}}& & \h^{i-j}(X,\Ksheaf^M_i)\ar[ru]\ar[rd] & \\
&&&\HH^{2i-j}(X,W^\dagger\Omega)\ar[uu]}$$
where the outer triangle leads to the desired diagram (\ref{DiagComp}).

In Section \ref{Cycle} we construct higher cycle classes using the method of Bloch \cite{Bloch}. For this we first recall the definition of Bloch's higher Chow groups $\CH^b(X,n)$, which under certain conditions calculate Voevodsky's motivic cohomology. They form together the higher Chow ring $\CH^\ast(X,\cdot)$ of $X$ and Bloch establishes further properties useful for a cohomology theory, among other things there is a rational relation with algebraic $K$-theory which motivates the construction of higher cycle class maps. Similar to the method used for the Chern classes, we construct first cycle maps for the Milnor $K$-sheaf
$$\eta_M^{bn}:\CH^b(X,n)\rightarrow \h^{b-n}(X,\Ksheaf_b^M),$$
which satisfy a normalisation property, allow flat pull-backs and are compatible with addition and multiplication thus giving a homomorphism of rings. This can be done because the target cohomology theory $\h^n(X,\Ksheaf_b^M)$ satisfies certain properties such as weak purity. Lastly, we use again the map 
$$d\log^i: \Ksheaf^M_i\rightarrow W^\dagger\Omega[i]$$
to obtain morphisms of higher cycle classes
$$\eta_{\text{sc}}^{bn}:\CH^b(X,n)\rightarrow \HH^{2b-n}(X,W^\dagger\Omega^{\geqslant b}).$$

\subsection*{Acknowledgements}

The work on this subject began as my PhD thesis at the University of Utah and this article is a modified version of part of it. I would like to thank my advisor Wies\l{}awa Nizio\l{} for introducing me to the subject as well as all the advice and help that I received along the way. Part of this work was done at the \'Ecole Normale Sup\'erieure in Lyon and I am very grateful for their support and hospitality. I also would like to thank Christopher Davis and Lance Miller for very helpful and inspiring discussions.

\section{Preliminaries}

In this paper, let $p$ be a fixed prime. For a scheme $X$ and a closed point $x$, let $\kappa(x)$ be the associated residue field. The generic point is denoted by $\xi$. 

\subsection{Milnor $K$-theory}\label{Milnor}

In this section we recall the definition and basic properties of Milnor $K$-theory for fields and rings with infinite and finite residue field respectively following \cite{Kerz2} and assure the validity of the definition on the \'etale site as well. 

\subsubsection*{The theory for local rings with infinite residue fields}

Milnor $K$-theory for fields is well-known and studied. Now one might attmept to define the following. 

\begin{defi} For a unital ring $R$ let
$$\overline{K}_\ast^M(R)=T^\ast(R)\slash J$$
where $J$ is the two-sided homogeneous ideal generated by the Steinberg relations and elements of the form $a\otimes(-a)$.
\end{defi}

On the other hand the following makes sense. 

\begin{defi}\label{KRingRegSemiLoc} For a regular semi-local ring $R$ over a field $k$ the Milnor $K$-groups are given by
$$K_n^M(R)=\Ker\left(\bigoplus_{x\in R^{(0)}}K_n^M(k(x))\xrightarrow{\partial}\bigoplus_{y\in R^{(1)}}K_n^M(k(y))\right).$$
\end{defi}

If $R$ is a regular semi-local ring over a field, there is a canonical homomorphism of groups
$$\overline{K}_i^M(R)\rightarrow K_i^M(R)$$
which is surjective if the base field is infinite (or sufficiently large, as in \cite{Kerz2}).

This definition globalises to schemes.

\begin{defi}\label{Def1.5} Define $\overline{\Ksheaf}_\ast^M$ to be the Zariski sheaf associated to the presheaf
$$U\mapsto\overline{K}_\ast^M(\Gamma(U,\OO_U))$$
on the category of schemes.\end{defi}

Inspired by Definition \ref{KRingRegSemiLoc} one defines the following.

\begin{defi}\label{Defi19} Let $\Ksheaf_n^M$ be the sheaf
$$U\mapsto \Ker\left(\bigoplus_{x\in U^{(0)}}i_{x\ast}K_n^M(k(x))\xrightarrow{\partial}\bigoplus_{y\in U^{(1)}}i_{y\ast}K_n^M(k(y))\right)$$
on the big Zariski site of regular varieties (schemes of finite type) over a field $k$, where $i_x$ is the embedding of a point $x$ in $U$.\end{defi}

One part of the Gersten conjecture for Milnor $K$-theory is to show that these two definitions coincide. Kato constructed a Gersten complex of Zariski sheaves for Milnor $K$-theory of a scheme $X$
\begin{equation}\label{GerstenComplex} 0\rightarrow\overline{\Ksheaf}_n^M\rightarrow\bigoplus_{x\in X^{(0)}}i_{x\ast}K_n^M(k(x))\rightarrow\bigoplus_{y\in X^{(1)}}i_{y\ast}K_n^M(k(y))\rightarrow\cdots\end{equation}

In \cite{Rost} Rost gives a proof that this sequence is exact if $X$ is regular and of algebraic type over an arbitrary field $k$ except possibly at the first two places. Exactness at the second place was shown independently by Gabber and Elbaz--Vincent/M\"uller--Stach. Finally Kerz proved that the Gersten complex is exact at the first place for $X$ a regular scheme over a field, such that all residue fields are ``big enough''. Hence the Gersten conjecture holds for Milnor $K$-theory in this case.

\begin{cor}\label{CorGersten} Let $X$ be a regular scheme of dimension $n$ over a field with enough elements. Then
$$\Ksheaf_\ast^M=\overline{\Ksheaf}_\ast^M.$$
\end{cor}

\subsubsection*{The theory for local rings with finite residue fields}\label{ImprMK}

As Kerz points out in \cite{Kerz3}, the Gersten conjecture does not hold in general if one uses the same construction of Milnor $K$-theory for local rings with finite residue fields.

Let $\FaiAb$ be the category of abelian sheaves on the big Zariski site of schemes and $\Fan$ the full subcategory of sheaves that admit a transfer map in the sense of Kerz \cite{Kerz3}.  Furthermore, let $\Fan^\infty$ be the full subcategory of sheaves in $\FaiAb$ which admit norms as described if we restrict the system to local $A$-algebras $A'$ with infinite (or ``big enough'',cf. \cite{Kerz2}) residue fields.

\begin{ex} Kerz shows that the Milnor $K$-sheaf $\overline{\Ksheaf}^M_\ast$ is continuous and an element of $\Fan^\infty$ (cf. \cite[Proposition 4]{Kerz3}). \end{ex}

A main result in Kerz's article \cite{Kerz3} is that for a continuous functor $F\in\Fan^\infty$ there exists a continuous functor $\widehat{F}\in\Fan$ and a natural transformation satisfying a universal property. Namely, for an arbitrary continuous functor $G\in\Fan$ together with a natural transformation $F\rightarrow G$ there is a unique natural transformation $\widehat{F}\rightarrow G$ making the diagram
$$
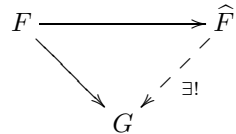
\xymatrix{F\ar[rr]\ar[rd] &&\widehat{F}\ar@{-->}[ld]^{\exists!}\\&G&}$$
commutative. Moreover, for a local ring with infinite residue field, the two functors coincide. 

A corollary is the existence of an improved Milnor $K$-theory, taking into account that $\overline{\Ksheaf}_n^M$ is in $\Fan^\infty$ and continuous.

\begin{cor}\label{KtoFunctor} For every $n\in\NN$ there exists a universal continuous functor $\widehat{\Ksheaf}_n^M\in\Fan$ and a natural transformation
$$\overline{\Ksheaf}_n^M\mapsto \widehat{\Ksheaf}_n^M$$
such that for any continuous $G\in\Fan$ together with a natural transformation $\overline{\Ksheaf}_n^M\rightarrow G$ there is a unique natural transformation $\widehat{\Ksheaf}_n^M\rightarrow G$ such that the diagram
$$
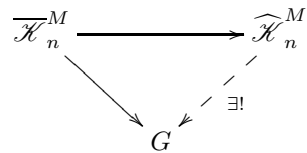
\xymatrix{\overline{\Ksheaf}_n^M\ar[rr]\ar[dr] &&\widehat{\Ksheaf}_n^M\ar@{-->}[dl]^{\exists!}\\&G&}$$
commutes.\end{cor}
Kerz proves in \cite[Proposition 10]{Kerz3} that $\widehat{\Ksheaf}_\ast^M$ satisfies the Gersten conjecture, and as before we deduce the

\begin{cor}\label{CorGersten2} Let $X$ be a smooth scheme with finite residue fields. Then
$$\Ksheaf_\ast^M=\widehat{\Ksheaf}_\ast^M$$
where $\Ksheaf_\ast^M$ is as in Definition \ref{Defi19}.
\end{cor}

Another important feature of the improved Milnor $K$-theory is that it is locally generated by symbols. In other words, it's elements satisfy the Steinberg relation \cite[Theorem 13]{Kerz2}

\subsubsection*{Milnor $K$-theory on the \'etale site}\label{MilnorKOnEtale}

Although the improved Milnor $K$-sheaf was constructed on the big Zariski site of all schemes, we can consider it as a sheaf over the big \'etale site. In particular, the theory makes sense on the small \'etale site $X_{\text{\'et}}$ of a scheme $X$.

More precisely, we can define $\overline{\Ksheaf}_\ast^M$ on the big \'etale site as in Definition \ref{Def1.5} \'etale locally instead of Zariski locally. Let $\FaiAb_{\text{\'et}}$, $\Fan_{\text{\'et}}$ and $\Fan_{\text{\'et}}^\infty$ the \'etale analogues of the above defined categories. This still makes sense as everything is only defined and described locally. On a similar note, continuity can be defined locally, so that the Milnor $K$-sheaf over the \'etale site is also continuous. The theorem now reads:

\begin{theo}\label{ImprEt} For a continuous functor $F\in\Fan^\infty_{\text{\'et}}$ there exists a universal continuous functor $\widehat{F}\in\Fan_{\text{\'et}}$ and a natural transformation $F\rightarrow \widehat{F}$. That means, for an arbitrary continuous functor $G\in\Fan_{\text{\'et}}$ together with a natural transformation $F\rightarrow G$ there is a unique natural transformation $\widehat{F}\rightarrow G$ making the diagram
$$
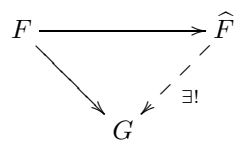
\xymatrix{F\ar[rr]\ar[rd] &&\widehat{F}\ar@{-->}[ld]^{\exists!}\\&G&}$$
commutative. Moreover, for a local ring with infinite residue field, the two functors coincide.\end{theo}
\begin{proof}: The proof of the Zariski case used in \cite[Theorem 7]{Kerz3}  is purely local. Taking into account that the functors in question are sheafifications of presheaves on the category of local rings and furthermore that the only condition which goes beyond this is the existence of a transfer map for finite \'etale extensions of local rings, we see that the arguments can be carried over verbatim to the case of the \'etale site instead of the Zariski site. In fact, this is valid for any (Grothendieck) topology in between the \'etale and Zariski topology.   \end{proof}
\medskip

The properties of the improved Milnor $K$-sheave from \cite{Kerz3} in particular  \cite[Proposition 10]{Kerz3} cited above hold in the case of the \'etale site equally. In particular, the Gersten complex
$$0\rightarrow \widehat{K}^M_n(A)\rightarrow K_n^M(F)\rightarrow \oplus_{x\in X^{(1)}}K_{n-1}^M(k(x))\rightarrow\cdots$$
is exact and the improved Milnor $K$-sheaf over the \'etale site is locally generated by symbols.

\subsection{Cycle modules}\label{Rost}

\subsubsection*{Definition and properties}

We recall the definition of cycle modules given by Rost.

Rost defines in \cite{Rost} a cycle premodule $M$ as a functor from the category of fields to the category of modules over Milnor $K$-theory which have transfer morphisms and residue maps for discrete valuations and satisfy the the usual canonical axioms, among other things, it has a $\ZZ$-grading compatible with Milnor $K$-theory. 

A ring structure on a cycle premodule $M$ is a pairing $M\times M\rightarrow M$ which respects grading inducing for each $F$ an associative and anti-commutative ring structure.

A cycle module $M$ over a field $k$ is a cycle premodule with two additional conditions which allow to define a differential map $d$ on the associated cycle module and guarantee that $d\circ d=0$. Furthermore, Rost shows that cycle modules have the homotopy property and permit proper push forwards. 

\subsubsection*{Cycle complexes}

We distinguish dimension and codimension complexes.

\begin{defi} For a cycle module $M$ over $X$ we define a complex of graded modules with respect to dimension
$$C_p(X;M) = \coprod_{x\in X_{(p)}}M(x)\, ,\qquad d = d_X: C_p(X;M)\rightarrow C_{p-1}(X;M),$$
where $d$ is induced by the residue map, and similarly with respect to codimension
$$C^p(X;M) = \coprod_{x\in X^{(p)}}M(x)\, , \qquad d = d_X: C^p(X;M)\rightarrow C^{p+1}(X;M).$$
\end{defi}
 We can also define a version of the above with support in a closed subscheme $Y\rightarrow X$. Then define
$$C_p^Y(X;M)=\coprod_{\substack{x\in X_{(p)}\\x\in Y}}M(x)\qquad\text{and}\qquad C^p_Y(X;M)=\coprod_{\substack{x\in X^{(p)}\\x\in Y}}M(x).$$

It is not hard to show that $d\circ d=0$ in both cases so that these are indeed complexes. There are four relevant types of maps between cycle complexes. A finite morphism of schemes $f:X\rightarrow Y$ induces a push-forward
$$f_\ast: C_p(X;M)\rightarrow C_p(Y;M).$$
If $f$ has relative dimension $s$, we may also define a pull-back map
$$f^\ast:C_p(Y;M)\rightarrow C_{p+s}(X;M).$$
The fact, that cycle modules are modules over the Milnor $K$-ring makes it possible to define for each section $\{a_1,\ldots,a_n\}$ for $a_i\in\OO_X^\ast$ a map
\begin{equation}\label{SubsecCyclecom}\{a_1,\ldots,a_n\}:C_p(X;M)\rightarrow C_p(X;M)\end{equation}
which we call multiplication with units. Finally, let $X$ be of finite type over a field, $i:Y\rightarrow X$ a closed immersion and $j:U=X\backslash Y\rightarrow X$ the inclusion of the complement. We define the boundary map associated to a so-called boundary triple $(Y,i,X,j,U)$
$$\partial=\partial_Y^U:C_p(U;M)\rightarrow C_{p-1}(Y;M).$$
Although we just defined the four basic maps for the cycle complex with respect to dimension, it is clear that there are similar maps for the codimension-cycle complex (or co-cycle complex). Sums of composites of maps of the four basic types are called generalised correspondences. Rost shows in \cite[Section 4]{Rost} compatibilities of the four basic types of maps as desired for a reasonable cycle theory. 

The grading on $M$ induces a grading on the dimension and codimension complex via
$$C_p(X;M,q) = \coprod_{x\in X_{(p)}}M_{q+p}(x)\qquad\text{and}\qquad C^p(X;M,q) = \coprod_{x\in X^{(p)}}M_{q-p}(x).$$
The maps to be considered will respect the grading.

\subsubsection*{The cohomology of cycle modules}

The following results are useful if one is faced with the task to calculate the cohomology of a cycle module explicitly.

\begin{defi} The Chow group of $p$-dimensional (co)cycles with coefficients in $M$ without and with is defined as the $p$\textsuperscript{th} homology group of the above  complexes
\begin{eqnarray*}A_p(X;M):=\h_p(C_\ast(X;M))\quad&\text{and}&\quad A^p(X;M) := \h^p(C^\ast(X;M)),\\
A_p(X;M,n) := \h_p(C_\ast(X;M,n))\quad &\text{and}&\quad A^p(X;M,n) := \h^p(C^\ast(X;M,n)).\end{eqnarray*}
\end{defi}

The morphisms induced by the four basic maps on the cycle complexes induce maps on the homology and cohomology groups and commute, respectively anti-commute with the differentials. The compatibilities carry over from cycle modules to Chow groups (for proper $f$, $f'$ and flat $g$). Moreover, a boundary triple $(Y,i,X,j,U)$ induces a long exact sequence for homology
\begin{equation}\label{LESfH}\cdots\xrightarrow{\partial}A_p(Y;M)\xrightarrow{i_\ast}A_p(X;M)\xrightarrow{j_\ast}A_p(U;M)\xrightarrow{\partial}A_{p-1}(Y;M)\xrightarrow{i_\ast}\cdots.\end{equation}

An interesting feature of the Chow groups as defined here, which brings it closer to classical topology is the \textbf{homotopy invariance}. Let $\pi:V\rightarrow X$ be an affine bundle of dimension $n$. Then
\begin{equation}\label{HIH}\pi^\ast:A_p(X;M)\rightarrow A_{p+n}(V;M)\end{equation}
is bijective for all $p$. If $X$ is equidimensional, then
\begin{equation}\label{HIC}\pi^\ast:A^p(X;M)\rightarrow A^p(V;M)\end{equation}
is bijective for all $p$. In particular this applies to fiberproducts with $\aA^n$. Rost proves this in \cite[Proposition 8.6]{Rost} using a spectral sequence argument.

We now come to one of the main results which we need from Rost's discussion for our purpose.

Let $M$ be a cycle module over a field $k$.

\begin{theo}\label{Theo61} Let $X$ be smooth, semi-local and a localisation of a separated scheme of finite type over $k$. Then
$$A^p(X;M)=0\qquad\text{ for }\qquad p>0.$$
\end{theo}

In other words, the complex $C^\ast(X;M)$ is acyclic. It is clear that one deduces a similar statement from this for the graded complexes.

When $X$ is smooth, it is possible to sheafify the notion of cycle modules as follows.

\begin{defi} Let $\MM_X$ (resp. $\MM_q$) be the sheaf on $X$ given by
$$U\mapsto A^0(U;M)\subset M(\xi_X),\qquad \left(\text{ resp. }\quad U\mapsto A^0(U;M,q)\subset M_q(\xi_X)\quad\right) $$
\end{defi}

In the case, when $M=K_\ast^M$ is Milnor $K$-theory, this definition coincides with the Milnor $K$-sheaf as defined in \ref{Defi19}.

\begin{cor}\label{Cor65} For a smooth variety $X$ over $k$ there are natural isomorphisms
$$A^p(X;M)=\h^p(X,\MM_X)\quad\text{ and }\quad A^p(X;M,q)=\h^p(X,\MM_q).$$
\end{cor}

This tells us, that we can calculate the cohomology of the Milnor $K$-sheaf $\Ksheaf_\ast^M$ in terms of the cohomology of the associated cycle complex.

\subsubsection*{Projective bundle formula for Chow groups}

Following Gillet's axiomatic framework to construct Chern classes, one of the main steps is to establish a projective bundle formula. We give here a more detailed proof of the sketch in \cite[Proposition 54]{Gillet2}

\begin{prop}\label{ProjectiveBundleFormula} Let $M$ be a cycle module, $X$ a variety over $k$ and $\pi:\edg\rightarrow X$ a vector bundle of constant rank $n$. Let further $\xi\in\h^1(\PP(\edg),\OO^\ast)$ be the class of $\OO(1)$. Then there is an isomorphism
$$A^p(\PP(\edg),M,q)\cong\bigoplus_{i=0}^{n-1}A^{p-i}(X,M,q-i)\xi^i.$$
\end{prop}
\begin{proof}: This being a local question, we may assume without loss of generality that $X=\Spec A$ is affine and that $\edg=\OO_X^n$. The first part of the proof establishes the result for the case of a point $X=\Spec k$. From there, the second part deduces the general result.

Now let $X=\Spec k$ be a point. This implies in particular that $\PP(\edg)=\PP^n_k$. The formula we have to show in this case reads
$$A^p(\PP^n,M,q)=A^0(X,M,q-p)\xi^p$$
because obviously the higher Chow groups vanish for a point so that we are left with only one term with $i=p$. Let $j:\PP^{n-1}\subset\PP^n$ be the hyperplane at infinity, $\aA^n$ its complement and $i:\aA^n\rightarrow \PP^n$ the inclusion of the open subset. Recalling the definition of cycle complexes as
\begin{eqnarray*} C^{p-1}(\PP^{n-1};M,q-1)&=&\prod_{x\in (\PP^n)^{(p-1)}}M_{q-1-(p-1)}(x)=\prod_{x\in (\PP^n)^{(p-1)}}M_{q-p}(x)\\
C^p(\PP^n;M,q) &=&  \prod_{x\in (\PP^n)^{(p)}}M_{q-p}(x)\\
C^p(\aA^n;M,q) &=&  \prod_{x\in (\aA^n)^{(p)}}M_{q-p}(x)\end{eqnarray*}
together with the fact that points of codimension $p-1$ in $\PP^{n-1}$ correspond to points of codimension $p$ in $\PP^n$ we see that the maps $i$ and $j$ induce via push-forward and pull-back respectively morphisms
$$j_\ast: C^{p-1}(\PP^{n-1};M,q-1)\rightarrow C^p(\PP^n;M,q)$$
and
$$i^\ast: C^p(\PP^n;M,q)\rightarrow C^p(\aA^n;M,q).$$
By choice and definition of $\PP^{n-1}$ and $\aA^n$, these morphisms of groups for varying $p$ combine to a short exact sequence of complexes
$$0\rightarrow C^\ast(\PP^{n-1};M,q-1)\left[1\right]\rightarrow C^\ast(\PP^n;M,q)\rightarrow C^\ast(\aA^n;M,q)\rightarrow 0$$
which gives rise to a long exact sequence of Chow groups by taking cohomology
$$\cdots\rightarrow A^{p-1}(\PP^{n-1};M,q-1)\xrightarrow{j_\ast}A^p(\PP^n;M,q)\xrightarrow{i^\ast}A^p(\aA^n;M,q)\rightarrow\cdots$$
where $j_\ast$ is the Gysin map. Let $\pi:\PP^n\rightarrow X=\Spec k$ be the projection induced from $\pi:\edg\rightarrow X$. Consequently the map associated to $\pi\cdot i$ on Chow groups
$$(\pi\cdot i)^\ast:A^p(\Spec(k),M,q)\rightarrow A^p(\aA^n,M,q)$$
is an isomorphism due to homotopy invariance (\ref{HIC}). Since the Chow groups of a point are trivial for $p>0$, the same holds true for the Chow groups of $\aA^n$. Therefore we can break up the long exact sequence. The first part reads
$$0\rightarrow A^0(\PP^n;M,q)\xrightarrow{i^\ast} A^0(\Spec k,M,q)\rightarrow A^0(\PP^n-1;M,q-1)\xrightarrow{j_\ast} A^1(\PP^n,M,q)\rightarrow 0.$$
Per definitionem
$$A^0(\Spec k,M,q)=C^0(\Spec k;M,q)=M_q(k)$$
and
$$A^0(\PP^n, M,q)=\Ker\left( C^0(\PP^n;M,q)\rightarrow C^1(\PP^n;M,q)\right)= \Ker\left( \prod_{x\in (\PP^n)^0}M_q(x)\rightarrow  \prod_{x\in (\PP^n)^1}M_{q-1}(x)\right)$$
and the fact that $(\pi\cdot i)^\ast=i^\ast\cdot \pi^\ast$ is an isomorphism shows that $\pi^\ast$ splits the sequence as $i^\ast$ is injective. Thus
$$A^0(\PP^n,M,q)\cong A^0(\aA^n,M,q).$$
The other parts of the long exact sequence become for every $i\geqslant 1$
$$0\rightarrow A^{i-1}(\PP^{n-1},M,q-1)\xrightarrow{j_\ast} A^i(\PP^n,M,q)\rightarrow 0$$
hence we have isomorphisms, where the map $j_\ast$ is the same as cap product with $\xi$. By induction it follows that the natural map
$$\xi^p:A^0(\Spec k,M,q-p)\rightarrow A^p(\PP^n,M,q)$$
is an isomorphism.

Assume now that $X=\Spec A$ for a $k$-algebra $A$. The projective bundle $\PP(\edg)$ takes the form $\PP^n_X=\PP^n_k\times X$, and it is useful to keep the following commutative diagram of the fibre product in mind
\begin{equation}\label{FiberDiag}\xymatrix{\PP^n_X\ar[r]^{f'}\ar[d]_{\pi'} & \PP^n_k\ar[d]^{\pi}\\ X\ar[r]_f & \Spec k}\end{equation}
Let $\xi$ be again the image of the twisting sheaf $\OO_{\PP^n}(1)$. Cup product with $\xi^i$ for $0\leqslant i\leqslant n-1$ provides a natural morphism of cohomology
$$\oplus_{i=0}^{n-1}\xi^i:\bigoplus_{i=0}^{n-1}A^{p-i}(X,M,q-i)\rightarrow A^p(\PP(\edg),M,q),$$
and the task is to show that this is an isomorphism. To this end we use the fact that $A^p(X;M,q)=\h^p(X,\MM_q)$. Note that $\h^p(\PP_X,\MM_q)=\R^p\Gamma_{\PP_X^n}\MM_q$ and $\h^{p-i}(X,\MM_{q-i})=\R^{p-i}\Gamma_X\circ\R^{i}\pi'_\ast(\MM_q)$. Thus the morphism above given by successive multiplication with $\xi^i$'s induces a morphism
$$\DF\Gamma_X\DF\pi'_\ast\MM_q\rightarrow\DF\Gamma_{\PP_X^n}\MM_q$$
in the derived category of abelian groups. Since $\Gamma_X=\Gamma_k\circ f_\ast$ where we write for simplicity $\Gamma_k=\Gamma_{\Spec k}$, there is a spectral sequence
$$\R^i\Gamma_k\circ\R^j f_\ast \Rightarrow R^n\Gamma_X$$
which degenerates because $\R^i\Gamma_k=0$ if $i\neq 0$. Therefore there is an isomorphism
$$\oplus_{i+j=p}\R^i\Gamma_X\circ\R^j\pi'_\ast(\MM_q)\cong\oplus_{i+j=p}\Gamma_k(\R^i f_\ast\circ R^j\pi'_\ast)(\MM_q),$$
which is in the derived category
$$\DF\Gamma_X\DF\pi'_\ast(\MM_q)\cong\Gamma_k\DF f_\ast\DF\pi'_\ast(\MM_q).$$
For the derived functors of the compositions $f_\ast\circ\pi'_\ast$ and $\pi_\ast\circ f'_\ast$ there are as usual two  spectral sequences
$$\R^i f_\ast \R^j\pi'_\ast\Rightarrow \R^n(f_\ast\circ\pi'_\ast)\;\text{ and }\; \R^i\pi_\ast \R^j f'_\ast\Rightarrow\R^n(\pi_\ast\circ f'_\ast),$$
yet the commutativity of the diagram (\ref{FiberDiag}) implies that they converge in fact to the same object. This in turn leads to an isomorphism in the derived category
$$\Gamma_k\DF f_\ast\DF\pi'_\ast(\MM_q)\cong\Gamma_k\DF\pi_\ast\DF f'_\ast(\MM_q).$$
If we recall that push-forward is well defined for cycle modules and compatible with the structure and therefore transforms cycle modules into cycle modules, we see that by the result of the first part of the proof the right-hand-side of this is isomorphic to
$$\DF\Gamma_{\PP^n_k}\DF f'_\ast(\MM_q).$$
Now similarly to above, the spectral sequence associated to the equality of functors $\Gamma_{\PP^n_X}=\Gamma_{\PP^n_k}\circ f'_\ast$ induces an isomorphism in the derived category
$$\DF\Gamma_{\PP^n_k}\DF f'_\ast(\MM_q)\cong\DF\Gamma_{\PP^n_X}\MM_q.$$
Putting everything together yields an isomorphism
$$\DF\Gamma_X\DF\pi'_\ast(\MM_q)\cong\DF\Gamma_{\PP_X^n}\MM_q$$
which corresponds by construction exactly the morphism of cohomology introduced at the beginning by cap product with $\xi^i$'s.  \end{proof}
\medskip

\subsection{Chern classes for higher algebraic $K$-theory with coefficients in the Milnor $K$-sheaf}\label{Gillet}

In this section, we establish the machinery to construct Chern classes.

\subsubsection*{Chern classes with coefficients in a generalised cohomology theory}\label{subsetChernGillet}

The idea behind Gillet's generalised cohomology theory in \cite{Gillet}, that concerns this subject, is that a cohomology theory with certain propertis allows for the construction of universal classes over the classifying space $B.\GL_n$., which in turn yield compatible universal classes for $\GL_n$. Using the Dold--Puppe functor, one obtains the desired characteristic classes.

Essentially, such a cohomology theory $\underline{\Gamma}^\ast(\ast)$ is given as a graded complex of sheaves of abelian groups together with an associative, graded-commutative pairing with unit in the derived category
$$\underline{\Gamma}^\ast(\ast)\otimes^L_{\ZZ}\underline{\Gamma}^\ast(\ast)\rightarrow\underline{\Gamma}^\ast(\ast).$$
Additionally, one requires it to satisfy eleven axioms, whereof we mention only the existence of a cap product, a projective bundle formula and the existence of a cycle class map, as these are the ones needed for the construction of Chern classes.

\begin{defi}\label{TCC} A theory of Chern classes with coefficients in $\Gamma$ for representations of sheaves of groups assigns for each $X$ to any representation $\rho:\GD\rightarrow\GL(\F)$ on a locally free $\OO_X$-module $\F$ classes
$$C_i(\rho)\in\h^{di}(X,\GD,\Gamma(i))$$
where $d\in\{1,2\}$ depends on the chosen duality theory $\Gamma$. These classes satisfy the following axioms.
\begin{enumerate}\item\label{TCCFunctoriality} \textbf{Functoriality.} Let $f:X\rightarrow Y$ be a morphism of schemes and $\rho:\GD\rightarrow\GL(\F)$ a representation of sheaves of groups in $Y$ and $\varphi:\mathscr{H}\rightarrow f^\ast\GD$ a homomorphism of sheaves of groups on $X$. Moreover, let
$$f^\ast(\rho)\circ\varphi:\mathscr{H}\rightarrow\GL(\F\otimes_{\OO_Y}\OO_X)$$
be the induced representation on $X$. Then
$$C.(f^\ast(\rho)\circ\varphi)=\varphi^\ast(f^\ast(C.(\rho))).$$

\item\label{TCCWhitneySum} \textbf{Whitney sum formula or additivity.} Let
$$0\rightarrow (\rho',\F')\rightarrow (\rho,\F)\rightarrow (\rho'',\F'')\rightarrow 0$$
be an exact sequence of representations of $\GD$, then
$$C.(\rho)=C.(\rho')\cdot C.(\rho'').$$

\item\label{TCCTensor} \textbf{Tensor products.} Let $(\rho_1,\F_1)$ and $(\rho_2,\F_2)$ be representations of $\GD$ and $(\rho_1\otimes\rho_2,\F_1\otimes\F_2)$ their tensor product, then
$$\widetilde{C}.(\rho_1\otimes\rho_2)=\widetilde{C}.(\rho_1)\circledast C.(\rho_2),$$
where $\circledast$ is the product defined by the universal polynomials of Grothendieck and $\widetilde{C}.$ is the augmented total Chern class.

\item \textbf{Stability.} Let $\varepsilon:\{e\}\rightarrow\GL(\OO_X)\cong\OO_X^\ast$ be the trivial rank one representation. Then
$$C.(\varepsilon)=1.$$

\item \textbf{Normalisation.} For any representation $\rho:\GD\rightarrow\GL(\F)$ the zero class is trivial
$$C_0(\rho)=1.$$
\end{enumerate}
\end{defi}

Gillet shows in \cite{Gillet} that this definition is not void.

\begin{theo}\label{TheoChernClasses} Let $\Gamma$ be as before. Then for a category of schemes there is a theory of Chern classes with coefficients in $\Gamma$.\end{theo}

This is shown in three steps. The first one is to construct universal classes $C_{in}\in\h^{di}(B.\GL_n,\Gamma(i))$ with the help of the projective bundle formula using the cycle class map. The second step is to pass from the classes $C_i\in\h^{di}(B.\GL_n,\Gamma(i))$ to classes $C_i\in\h^{di}(X,\GL_n(\OO_X),\Gamma(i))$. At this point, one needs the cap product. Lastly, one generalises to any representation $\rho:\GD\rightarrow\Saut(\F)$ of a sheaf of groups on a scheme $X$ using functoriality. This forms the basis for his to conclude with the following theorem.

\begin{theo}\label{TheoHigherChernClasses} Let $\ver$ be a category of schemes over a fixed base, $X\in\ver$ and $\Gamma(\ast)$ a duality theory. Then there exists a theory of Chern classes for higher algebraic $K$-theory
$$c_{ij}:K_j(X)\rightarrow\h^{di-j}\left(X,\Gamma(i)\right).$$
\end{theo}

\subsubsection*{Higher Chern classes for the Milnor $K$-sheaf}

Let $X$ be smooth over a field $k$ of dimension $n$. Now set
$$\underline{\Gamma}_X^\ast(j)=\Ksheaf^M_j$$
for $j\geqslant 0$ and the zero sheaf otherwise, where this is seen as a complex with only one spot non-zero, and further let $S=k$. By the Gersten Conjecture (Corollary \ref{CorGersten})
$$\Ksheaf^M_\ast=\begin{cases}\overline{\Ksheaf}_\ast^M &\text{ in the case of infinite residue fields}\\
\widehat{\Ksheaf}_\ast^M & \text{ in the case of finite residue fields}\end{cases}$$
The associated generalised cohomology theory is
$$\h^i(X,\Gamma(j))=\h^i(C^\ast(X;K_\ast^M,j))=A^i(X;K_\ast^M,j).$$

\begin{theo}\label{TheoKsheafChern} There is a theory of Chern classes for vector bundles and higher algebraic $K$-theory of regular varieties over $k$ with infinite residue fields, with values in Zariski cohomology with coefficients in the Milnor $K$-sheaf:
$$c_{ij}^M:K_j(X)\rightarrow\h^{i-j}(X,\Ksheaf_i^M).$$
\end{theo}
\begin{proof}: We have to verify that the duality theory associated with $\underline{\Gamma}_X^\ast(j)=C^\ast(X;K_\ast^M,j)$ satisfies the above mentioned necessary axioms. 
\begin{enumerate} \item\label{CapProdMilnorK} \textbf{Cap product.} Recall that there is a pairing of cycle modules
$$K_\ast^M\times K_\ast^M\rightarrow K_\ast^M$$
which respects grading. Using the ``multiplication-with-units'' map (\ref{SubsecCyclecom}) from in Subsection \ref{Rost} this induces a pairing of complexes
$$C_p(X,K_\ast^M,j)\times C^q_Y(X,K_\ast^m,i)\rightarrow C_{p-q}(Y,K_\ast^M,j-i),$$
where we have used that $C^q=C_{n-q}$ as X is of dimension $n$ and where $C^q_Y$ means sections with support in $Y$. This map respects the grading on $K_\ast^M$ since the original pairing on $K_\ast^M$ does so. Moreover, it respects the grading in dimension as it is a generalised correspondence map mentioned in \cite[(3.9)]{Rost}. Applying the (co)homology functor, we obtain a pairing
$$\bigcap:A_p(X;\K_\ast^M,j)\otimes A^q_Y(X;K_\ast^M,i)\rightarrow A_{p-q}(Y;K^M_\ast,{j-i}).$$

\item \textbf{Projective bundle formula.} We proved this in Proposition \ref{ProjectiveBundleFormula}

\item \textbf{Cycle class map.} This is clear from the definition of the first Milnor $K$-group. Indeed, recall that by definition of the Milnor $K$-sheaf
$$\Ksheaf^M_1=\OO_X^\ast,$$
and the isomorphism for a scheme $X$
$$\Pic(X)\cong \h^1(X,\OO_X^\ast)$$
gives a natural transformation of contravariant functors on the big Zariski site $\ver$.

\end{enumerate}
Now we can apply Theorem \ref{TheoChernClasses} and Theorem \ref{TheoHigherChernClasses} and go through the construction sketched earlier to obtain the claim.  \end{proof}
\medskip

\begin{rem} As we discuss in Section \ref{MilnorKOnEtale} this translates one-to-one to the \'etale case, and we can replace in the theorem Zariski cohomology with \'etale cohomology, with which we want to work. \end{rem}

\subsection{Logarithmic Witt differentials}\label{Logarithmic}

Let $k$ be a perfect field of characteristic $p>0$. For a smooth $k$-scheme $X$ let $W^\dagger \Omega_X$ be the overconvergent de Rham--Witt complex as defined by Davis, Langer and Zink in \cite{DavisLangerZink} and $W\Omega_X$ the usual de Rham--Witt complex \cite{Illusie}. The goal of this section is to find a good notion of logarithmic overconvergent differentials and prove that they factor through Milnor $K$-theory. 

\subsubsection*{Definition of log-differentials}

For $n\in\NN$ denote
$$d\log:\OO_X^\ast\rightarrow W_n\Omega_X^1$$
the morphism of abelian sheaves defined locally by $x\mapsto\frac{d\left[x\right]}{\left[x\right]}$. This induces a morphism of projective systems
$$d\log:\OO_X^\ast\rightarrow W_\bullet\Omega_X^1.$$
Let $W_n\Omega^i_{X,\log}\subset W_n\Omega^i_X$ be the subsheaf generated \'etale-locally by sections of the form $d\log\left[x_1\right]\ldots d\log\left[x_i\right]$ for $x_j\in\OO_X^\ast$. This construction is functorial in $X$, and the product structure of $W_n\Omega^\bullet_X$ carries over to $W_n\Omega_{X,\log}^\bullet$. For $n$ variable, $W_\bullet\Omega^\bullet_{X,\log}$ is an abelian subprosheaf of $W_\bullet\Omega^\bullet_X$ and we set $W\Omega_{X,\log}^\bullet:=\varprojlim W_\bullet\Omega^\bullet_{X,\log}$. For $i\in\NN_0$ there is a short exact sequence of prosystems for \'etale topology
$$0\rightarrow W_\bullet\Omega_{X,\log}^i\rightarrow W_\bullet\Omega_X^i\xrightarrow{\f-1} W_\bullet\Omega_X^i\rightarrow 0$$
where $\f$ denotes a lift of the Frobenius endomorphism. 

Taking the limit yields an exact sequence
$$0\rightarrow W\Omega_{X,\log}^i\rightarrow W\Omega_X^i\xrightarrow{\f-1}W\Omega_X^i.$$
This means that $W\Omega_{X,\log}^\bullet=\Ker(\f-1)\subset W\Omega_X^\bullet$. It is not a priori clear that exactness holds also on the right, that is that $\f-1$ is surjective. 

Let $R_{nm}:W_n\Omega_X^\bullet\rightarrow W_m\Omega_X^\bullet$ be the restriction map for $n\geqslant m$. We want to show that for $i$ fixed the projective system $(W_n\Omega_{X,\log}^i, R_{nm})$ satisfies the Mittag--Leffler condition locally. Indeed, since \'etale locally $W_n\Omega^i_{X,\log}$ is generated by sections of the form $d\log\left[x_1\right]_n\ldots d\log\left[x_i\right]_n$, where the Teichm\"uller lifts are in truncated Witt vectors of length $n$, and the restriction maps $R_{nm}$ commute with multiplication, addition and differential, we have for $n\geqslant m\geqslant k$
$$R_{nk}\left(\frac{d\left[x\right]_n}{\left[x\right]_n}\right)=R_{mk}\left(\frac{d\left[x\right]_m}{\left[x\right]_m}\right)$$
and thus locally
$$R_{nk}\left(W_n\Omega^\bullet_{X,\log}\right)=R_{mk}\left(W_m\Omega^\bullet_{X,\log}\right).$$
This induces exactness of the sequence
$$0\rightarrow W\Omega_{X,\log}^i\rightarrow W\Omega_X^i\xrightarrow{\f-1}W\Omega_X^i\rightarrow 0$$
for \'etale topology (but not for global sections).

\subsubsection*{Basic Witt differentials}\label{SubsecBasic}

Assume now that $X$ is the spectrum of a polynomial algebra $A=k[X_1,\ldots,X_d]$. Langer and Zink proved \cite[Theorem 2.8]{LangerZink} that any element $\omega\in W\Omega_A^\bullet$ has a unique expression as a convergent sum of basic Witt differentials
$$\sum_{k,\PU}e(\xi_{k,\PU},k,\PU)$$
where $k$ runs over all possible weight functions and $\PU$ over all partitions of $\Supp k$ and for any $m$ $\xi_{k,\PU}\in {}^{\V^m}\!W(k)$ for almost all weights $k$. The last condition is another way of saying that the sum converges $p$-adically.

For fixed $k$ with a total ordering, let $\PU$ be a partition of $\Supp k=I_0\sqcup I_1\sqcup\ldots\sqcup I_\ell$ respecting the ordering. The interval $I_0$ may be empty, but the intervals $I_{1},\ldots,I_\ell$ not. A basic Witt differential $e=e(\xi,k,\PU)\in W\Omega_A^\ell$ of degree $\ell$ with $\xi= {}^{\V^{u(I)}}\! \eta\in {}^{\V^{u(I)}}\!W(k)$ is defined in the following way: Denote by $r\in\left[0,\ell-1\right]$ the first index such that $k_{I_{r+1}}$ is integral. Three cases occur.
\begin{enumerate}\item $I_0\neq \varnothing$, no condition on the integrality of $k$.
\begin{eqnarray*}e &=& {}^{\V^{u(I_0)}}\!\left(\eta \left[X\right]^{p^{u(I_0)}k_{I_0}}\right)\left(d{} ^{\V^{u(I_1)}}\!\left[X\right]^{p^{u(I_1)}k_{I_1}}\right)\cdots\left(d {}^{\V^{u(I_r)}}\!\left[X\right]^{p^{u(I_r)}k_{I_r}}\right)\\
 && \left( {}^{\f^{-t(I_{r+1})}}\!d\left[X\right]^{p^{t(I_{r+1})}k_{I_{r+1}}}\right)\cdots\left( {}^{\f^{-t(I_{\ell})}}\!d\left[X\right]^{p^{t(I_{\ell})}k_{I_{\ell}}}\right)\end{eqnarray*}
Here $\xi= ^{\V^{u(I_0)}}\eta$.
\item $I_0=\varnothing$ and $k$ not integral.
$$e=\left(d {}^{\V^{u(I_1)}}\!\left(\eta \left[X\right]^{p^{u(I_1)}k_{I_1}}\right)\right)\cdots\left(d {}^{\V^{u(I_r)}}\! \left[X\right]^{p^{u(I_r)}k_{I_r}}\right)\left( {}^{\f^{-t(I_{r+1})}}\!d\left[X\right]^{p^{t(I_{r+1})}k_{I_{r+1}}}\right)\cdots\left( {}^{\f^{-t(I_{\ell})}}\!d\left[X\right]^{p^{t(I_{\ell})}k_{I_{\ell}}}\right)$$
Similarly as before $\xi= ^{\V^{u(I_0)}}\eta$.
\item $I_0=\varnothing$ and $k$ integral.
$$e=\eta\left( {}^{\f^{-t(I_1)}}\!d\left[X\right]^{p^{t(I_1)}k_{I_1}}\right)\cdots\left({}^{\f^{-t(I_\ell)}}\!d\left[X\right]^{p^{t(I_\ell)}k_{I_\ell}}\right)$$
Here $\xi=\eta$.
\end{enumerate}

If in $e(\xi,k,\PU)$ the element $\xi$ is contained in  $^{\V^m}W(R)$ then the image of the basic Witt differential under the restriction map $R_m$ is trivial. 

\begin{prop}\label{Prop4.2.1} The action of $\f$, $\V$ and $\alpha\in W(k)$ on the basic Witt differentials are given as follows:
\begin{enumerate}\item $\alpha e(\xi,k,\PU)=e(\alpha\xi,k,\PU)$.
\item If $I_0\neq\varnothing$, or if $k$ is integral (first and third case above)
$${}^{\f}\! e(\xi,k,\PU)=e({}^{\f}\!\xi,pk,\PU)$$
\item If $I_0=\varnothing$ and $k$ is not integral (second case above)
$${}^{\f}\!e(\xi,k,\PU)=e( {}^{\V^{-1}}\!\xi,pk,\PU)$$
\item If $I_0\neq\varnothing$ or $k$ is integral and divisible by $p$
$${}^{\V}\!e(\xi,k,\PU)=e({}^{\V}\!\xi,\frac{1}{p}k,\PU)$$
\item If $I_0=\varnothing$ and $\frac{1}{p}k$ is not integral
$${}^{\V}\!e(\xi,k,\PU)=e(p {}^{\V}\!\xi,\frac{1}{p}k,\PU)$$
\end{enumerate}\end{prop}

This is Proposition 2.5 in \cite{LangerZink}. Note that if $\omega\in W\Omega_A$ is given as a unique decomposition in basic Witt differentials $\sum e$ then its image under Frobenius has the unique decomposition $\f\omega=\sum\f e$. In this sense one could say that the decomposition remains ``fixed'' under Frobenius. The types of basic Witt differentials are almost stable under the action of Frobenius, i.e. one could switch from type 2 to type 3, since the weight is multiplied by $p$, but this is the only switch from one type to another that can possibly occur. What is more, the Frobenius action is injective on the set of basic Witt differentials. 

\subsubsection*{The overconvergent de Rham--Witt complex}\label{deRhamWittsurcon}

Let $A$ be a polynomial algebra over $k$. We recall the definition of the overconvergent de Rham--Witt complex \cite{DavisLangerZink}. Let $\omega=\sum_{k,\PU}e(\xi,k,\PU)\in W\Omega_A$ given as its unique decomposition as a sum of basic Witt differentials. For $\varepsilon >0$ the Gau\ss{} norm is defined by
$$\gamma_{\varepsilon}(\omega)=\inf_{k,\PU}\left\{\ord_V\xi_{k,\PU}-\varepsilon|k|\right\}.$$

\begin{defi} An element $\omega=\sum_{k,\PU}e(\xi,k,\PU)\in W\Omega_A$ is said to be overconvergent of radius $\varepsilon$, if $\gamma_\varepsilon(\omega)>-\infty$.\end{defi}

Note that the Teichm\"uller lift of an element in $A$ is by default overconvergent.

Denote by $W^\varepsilon\Omega_A$ the overconvergent Witt differentials of radius $\varepsilon$. The overconvergent de Rham--Witt complex is the union over all possible constants $\varepsilon>0$
$$\bigcup_\varepsilon W^\varepsilon\Omega_A =: W^\dagger\Omega_A$$
which is a subdifferential graded algebra of $W\Omega_A$. 

If $A=k\left[t_1,\ldots,t_r\right]$ is a smooth finitely generated $k$-algebra, and $S$ the polynomial algebra $k\left[X_1,\ldots,X_r\right]$, then the morphism $S\rightarrow A$, $X_i\mapsto t_i$ induces a canonical epimorphism
$$\lambda:W\Omega_S\rightarrow W\Omega_A$$
of de Rham--Witt complexes. Set $W^\dagger\Omega_A=\lambda\left(W^\dagger\Omega_S\right)$. This does not depend on the presentation, although the radii of overconvergence do in general. 

Davis, Langer and Zink show, that this construction can be globalised to a smooth scheme for \'etale and Zariski topology \cite[Cor.1.7, Theo.1.8]{DavisLangerZink}. Thus for a (smooth) scheme $X$ we have a subcomplex of the classical de Rham--Witt complex
$$W^\dagger\Omega_X\subset W\Omega_X.$$

\begin{rem}\label{RemPn} Notice that by definition Witt differentials of finite length are overconvergent. Hence the natural morphism
$$W^\dagger\Omega_X\otimes\ZZ\slash p^n\ZZ\rightarrow W_n\Omega_X$$
is an isomorphism in the derived category of $\ZZ\slash p^n\ZZ$-modules. Indeed, 
$$W^\dagger\Omega^i_X\otimes\ZZ\slash p^n\ZZ\rightarrow W^i_n\Omega_X$$
is evidently an isomorphism for all $i\geqslant0$ of $\ZZ/p^n\ZZ$-modules. On the other hand, the $\Tor$-functor, 
$$\Tor_j(\ZZ/p^n\ZZ,W^\dagger\Omega_X^i)=0$$
for all $j>0$ and all $i\geqslant 0$ as multiplication by $p$ is injective in $W^\dagger\Omega_X^i$. 
Accordingly, there is an isomorphism
$$W\Omega_X\cong\varprojlim W^\dagger\Omega_X\otimes\ZZ\slash p^n\ZZ.$$
\end{rem}

\subsubsection*{Logarithmic differentials in the overconvergent de Rham--Witt complex}

Let $X$ be a smooth scheme over $k$. The map $d\log$ induces a morphism
\begin{eqnarray*}d\log: \OO_X^\ast &\rightarrow& W\Omega^1_{X,\log}\\
x &\mapsto& \frac{d\left[x\right]}{\left[x\right]}\end{eqnarray*}
defined \'etale locally. The aforementioned fact that Teichm\"uller lifts are overconvergent along with the fact that the overconvergent complex is a subdifferential graded algebra of the classical de Rham--Witt complex shows that the elements $d\log\left[x_1\right]\ldots d\log\left[x_i\right]$ for $x_j\in\OO_X^\ast$ are overconvergent. Therefore we have in fact a natural map
$$d\log:\OO_X^\ast \rightarrow W^\dagger\Omega_X\left[1\right].$$
Moreover, extending this to higher degrees yields for every $i\geqslant 0$ a morphism
\begin{eqnarray*}d\log^{\otimes i}: \OO_X^\ast\otimes\cdots\otimes \OO_X^\ast & \rightarrow & W\Omega^i_{X,\log}\rightarrow W^\dagger\Omega_X\left[i\right]\\
x_1\otimes\cdots\otimes x_i &\mapsto & d\log(x_1)\cdots d\log(x_i).\end{eqnarray*}

\begin{defi} We set $W^\dagger\Omega_{X,\log}$ to be the subcomplex of the overconvergent complex generated \'etale locally by logarithmic Witt differentials.\end{defi} 

\begin{rem}\label{RemLogPn}\begin{enumerate}\item In Section \ref{SubsecShortExactSequence} we show the equality
$$W^\dagger\Omega_{X,\log} = W\Omega_{X,\log},$$
where the second complex is the logarithmic subcomplex of the usual de Rahm-Witt complex $W\Omega_X$, as both can be realised as the kernel of the map $1-\f$.
\item As Gros points out in \cite[Th\'eor\`eme 1.3.3]{Gros} there is a natural isomorphism
$$W_\bullet\Omega_{X,\log}\otimes\ZZ\slash p^n\ZZ\xrightarrow{\sim} W_n\Omega_{X,\log}$$
in the derived category of $\ZZ\slash p^n\ZZ$-promodules and therefore a natural isomorphism
$$W\Omega_{X,\log}\otimes\ZZ\slash p^n\ZZ\xrightarrow{\sim}W_n\Omega_{X,\log}.$$
Moreover one has
$$W\Omega_{X,\log}\cong\varprojlim W\Omega_{X,\log}^i\otimes\ZZ\slash p^n\ZZ.$$
\end{enumerate}\end{rem}
\medskip

In order to relate this to the Milnor $K$-sheaf, we prove the following:

\begin{prop}\label{PropSteinberg} The symbols $d\log(x_1)\cdots d\log(x_i)$ with $x_1,\ldots x_i\in\OO_X^\ast$ satisfy the Steinberg relation.\end{prop}
\begin{proof}: Let $x\in\OO_X^\ast$ be a local section and assume that $1-x\in \OO_X^\ast$. It is enough to show locally that $d\log(x)d\log(1-x)=0$, or even that $d[x]d[1-x]=0$. 

Thus let $X=\Spec B$ where $B$ is a quotient of a polynomial algebra over $k$. For an element $b\in B$ we calculate the expression $d[b]d[1-b]$.  Consider the morphism of $k$-algebras
$$\psi: k[z]\rightarrow B$$
sending $z$ to $b$. This morphism induces by functoriality a morphism of differential graded algebras between the associated de Rham--Witt complexes,
$$\psi:W\Omega_{k_[z]}\rightarrow W\Omega_{B}$$
which by abuse of notation we also denote by $\psi$. Yet the de Rham--Witt complex of $k[z]$ is trivial in degree greater than 1,
$$W\Omega_{k[z]}: 0\rightarrow W\OO_{k[z]}\rightarrow W\Omega^1_{k[z]}\rightarrow 0,$$
thence is clear that $d[z]d[1-z]$ is zero. But as $\psi$ is a morphism of differentially graded algebras, this implies already that $d[b]d[1-b]$ as the image in $W\Omega_B$ is zero as well. 

From this follows that in the general case for $x\in\OO_X^\ast$
$$d\log(x)d\log(1-x)=\frac{d[x]}{[x]}\frac{d[1-x]}{[1-x]}=0.$$

Using basic properties of the de Rham--Witt complex as differentially graded algebra, in particular anti-commutativity, the claim follows. \end{proof}
\medskip

\begin{cor} Let $X$ be a smooth scheme over $k$ with infinite residue fields. For each $i$ the map
$$d\log^{\otimes i}: \OO_X^\ast\otimes\cdots\otimes \OO_X^\ast  \rightarrow  W^\dagger\Omega_X\left[i\right]$$
factor through $\Ksheaf_i^M$ on $X$ and $d\log^{\otimes i}$ can be augmented to a morphism of sheaves on $X$
$$d\log^i:\overline{\Ksheaf}_i^M\rightarrow W^\dagger\Omega\left[i\right].$$\end{cor}

The next step is to show that the overconvergent de Rham--Witt complex is an object of $\Fan_{\text{\'et}}$. 

\subsubsection*{The transfer map for the overconvergent complex}

We will now define a transfer for the overconvergent complex. Let $i:A\rightarrow B$ be a finite \'etale extension of local rings. Fix an explicit representation $B\cong A[T]\slash(f)$ with $f\in A[T]$ monic such that $f'$ is a unit. Let $G_{B/A}=\Aut_A(B)$ be the automorphism group of $B$ that fixes $A$. Since $B/A$ is \'etale, in particular unramified, and the extension is of finite degree, we know that 
$$\# G_{B/A}=\deg(B/A).$$
By functoriality of the overconvergent complex, each element $g\in G_{B/A}$ induces a morphism of complexes
$$g:W^\dagger\Omega_B\rightarrow W^\dagger\Omega_B.$$

\begin{lem} Let $g:B\rightarrow B$ be an automorphism that fixes $A$. Then the induced morphism of de Rham--Witt complexes is also an automorphism which fixes $W^\dagger\Omega_A$.\end{lem}
\begin{proof}: By functoriality of the (overconvergent) de Rham--Witt complex the induced morphism is an isomorphism. It fixes $W^\dagger\Omega_A$ because this is true for the usual de Rham--Witt complex by construction, thus the same holds true for the restriction to the overconvergent subcomplex. \end{proof}
\medskip

\begin{lem} Let $\omega\in W^\dagger\Omega_B$ be fixed by all $g\in G_{B/A}$. Then $\omega$ is in fact in $W^\dagger\Omega_A$. \end{lem}
\begin{proof}: Consider first the case when $x\in W(B)$. In this case, the elements of $G_{B/A}$ act component wise and the claim follows from the fact that $A$ is exactly the fixed ring of $G_{B/A}$. This is of course also true if we restrict to the overconvergent Witt vectors. 

To extend this results to the (overconvergent) de Rham--Witt complex, recall that there is an isomorphism
$$W\Omega_B\cong W(B)\otimes_{W(A)}W\Omega_A.$$
Compare the remark after Proposition 1.9 in \cite{DavisLangerZink}. It is clear that $W\Omega_A$ is fixed by the elements in $G_{B/A}$ so it comes down to the coefficients in $W(B)$. But for this ring we just showed the claim. Without difficulty this transfers over to the overconvergent subcomplex.  \end{proof}
\medskip

Define the following map
\begin{eqnarray*}N_{B/A}: W^\dagger\Omega_B &\rightarrow& W^\dagger\Omega_B\\
\omega &\mapsto& \sum_{g\in G_{B/A}}g\omega\end{eqnarray*}

\begin{prop} The map $N_{B/A}$ has image in $W^\dagger\Omega_A$ and the restriction to $W^\dagger\Omega_A$ is multiplication by $\deg(B/A)$:
$$N_{B/A}\circ i_\ast=\deg(B/A)\id_{W^\dagger\Omega_A}.$$
\end{prop}
\begin{proof}: Let $h\in G_{B/A}$ and $\omega \in W^\dagger\Omega_B$. Then
\begin{eqnarray*}h N_{B/A}(\omega) &=& h\sum_{g\in G_{B/A}}g\omega\\
 &=& \sum_{g\in G_{B/A}}hg\omega\\
 &=& \sum_{g'\in G_{B/A}}g'\omega = N_{B/A}(\omega).\end{eqnarray*}
Thus $N_{B/A}(\omega)$ is fixed by all elements of $G_{B/A}$ and by the previous lemma this means that $N_{B/A}(\omega)\in W^\dagger\Omega_A$. 

For the second part of the claim, let $\omega\in W^\dagger\Omega_A$. Therefore, $\omega$ is fixed by $G_{B/A}$, and 
$$N_{B/A}(\omega)=\sum_{g\in G_{B/A}}g\omega = (\# G_{B/A} )\omega.$$
But we have seen that $\# G_{B/A}=\deg(B/A)$. This proves the claim. \end{proof}
\medskip

This shows that it makes sense to call the defined map $N_{B/A}$ a norm for the overconvergent complex for $B$ over $A$.

\begin{cor} Let $i:A\rightarrow B$ as before and $A'$ a local $A$ algebra such that $B'=B\otimes_A A'$ is also local. Let $i':A'\rightarrow B'$ be the induced inclusion. Then the map $N_{B'/A'}$ is multiplicative with image in $W^\dagger\Omega_{A'}$ and
$$N_{B'/A'}\circ i_\ast'=\deg(B/A)\id_{W^\dagger\Omega_{A'}}.$$\end{cor}
\begin{proof}: This follows directly as $\deg(B'/A')=\deg(B/A)$. \end{proof}
\medskip

Moreover we have the following compatibility.

\begin{prop} Let $i:A\rightarrow B$ be as before. Let $f:A'\rightarrow A''$ be a morphism of local $A$-algebras such that $B'=B\otimes_A A'$ and $B''=B\otimes_A A''$ are also local. Denote by $f^B:B'\rightarrow B''$ the induced morphism. Then the diagram
$$\xymatrix{W^\dagger\Omega_{B'}\ar[r]\ar[d]_{N_{B'/A'}} & W^\dagger\Omega_{B''}\ar[d]^{N_{B''/A''}}\\
W^\dagger\Omega_{A'}\ar[r] & W^\dagger\Omega_{A''}}$$
commutes.\end{prop}
\begin{proof}: The ring extensions $B'/A'$ and $B''/A''$ are both finite \'etale of degree $\deg(B/A)$ since $B/A$ is finite \'etale and all are local rings. What is more, the corresponding automorphism groups are isomorphic
$$G_{B/A}\cong G_{B'/A'}\cong G_{B''/A''}.$$
Thus we see that for $\omega\in W^\dagger\Omega_{B'}$
\begin{eqnarray*} f_\ast\circ N_{B'/A'}(\omega) &=& f_\ast\sum_{g'\in G_{B'/A'}}g'\omega\\
&=& \sum_{g'\in G_{B'/A'}}f_\ast g'\omega\\
&=& \sum_{g''\in G_{B''/A''}}g''f^B_\ast\omega\\
&=& N_{B''/A''}\circ f^B_\ast(\omega).\end{eqnarray*}
Due to functoriality of $W^\dagger\Omega$ and by a similar statement for rings. \end{proof}
\medskip

Comparing what we have just shown with the definition of the category $\Fan_{\text{\'et}}$ in Section \ref{ImprMK} yields indeed the desired result. 

\begin{cor} The sheaf $W^\dagger\Omega$ is an object of $\Fan_{\text{\'et}}$.\end{cor}

\subsubsection*{Continuity of the overconvergent complex}

Continuity of a functor means that it commutes with filtering direct limits. The de Rham--Witt complex does not commute with general direct limits and is thus not continuous. Although the overconvergent one might be, we content ourselves to show that it commutes with direct limits of finite $k$-algebras. Since in our context only smooth schemes over $k$ appear, which are in particular locally of finite type, this is enough. 

\begin{lem} Let 
$$A=\varinjlim A_i$$
be a filtering direct limit in the category of finite $k$-algebras. Then for the functor of Witt vectors the natural homomorphism
$$\varinjlim W(A_i)\rightarrow W(A)$$
is an isomorphism.\end{lem}
\begin{proof}: Consider the ghost maps
$$w_i:W(A_i)\rightarrow A_i^{\NN}\qquad\text{ and }\qquad w:W(A)\rightarrow A.$$
These are by definition ring homomorphisms. Because of the fact that the ring structure on the image of the ghost map is defined component wise the natural map
$$\varinjlim(A_i^{\NN})\rightarrow\left(\varinjlim A_i\right)^{\NN}=A^{\NN}$$
is a monomorphism. Since we have only finitely many generators it is in fact an isomorphism. By definition of the Witt vectors the following diagram commutes
$$\xymatrix{\varinjlim W(A_i)\ar[r]\ar[d]_w & W(A)\ar[d]\\ \varinjlim \left(A_i^{\NN}\right)\ar[r]^\sim & A^{\NN}}$$
and the vertical maps are injective. Additionally, we have just seen, that the bottom line is an isomorphism. Thus it is clear that the top line is a monomorphism. Now take an element $a\in W(A)$. To see that it has a preimage in $\varinjlim W(A_i)$ project it down to $A^{\NN}$ via the ghost map. The image $w(a)$ has a preimage $\widetilde{w(a)}\in \varinjlim \left(A_i^{\NN}\right)$. However, as the element $w(a)$ comes from a Witt vector given in Witt components, it is possible by solving the corresponding equations recursively to recover an element $\widetilde{a}\in\varinjlim W(A_i)$ that maps to $a$. \end{proof}
\medskip

\begin{prop} The de Rham--Witt complex is continuous on the category of finite $k$-algebras.\end{prop}
\begin{proof}: We have to show that for a filtering direct limit of finite $k$-algebras 
$$A=\varinjlim A_i$$
the natural map
$$\varinjlim W\Omega_{A_i}\rightarrow W\Omega_{\varinjlim A_i}$$
is an isomorphism. We will show that $W\Omega_A$ satisfies the universal property of direct limits in the category of Witt complexes over $A$. Let $M$ be a Witt complex over $A$. This means that it is a differentially graded $W(A)$ algebra with morphisms $\f$ and $V$ which satisfy certain properties. Keeping in mind that by functoriality there is a morphism $W(A_i)\rightarrow W(A)$, we see that $M$ is also a Witt complex over each $A_i$. Therefore it makes sense to consider maps
$$W\Omega_{A_i}\rightarrow M.$$
In fact, for each $i$ there is exactly one map of this form, because the de Rham--Witt complex is the initial object in the category of Witt complexes over $A_i$. The same is true for the de Rham--Witt complex over $A$ in the category of Witt complexes over $A$. Thus there is exactly one map
$$W\Omega_A\rightarrow M.$$
What is more, this map is compatible with the ones over $A_i$, i.e. the diagrams
$$\xymatrix{W\Omega_{A_i}\ar[r]\ar[d]& M\ar[d]\\ W\Omega_A\ar[r] & M}$$
where the upper horizontal morphism consists of $W(A_i)$-algebras and the bottom one of $W(A)$-algebras, commutes. This follows from the fact that we have an isomorphism $\varinjlim W(A_i)\xrightarrow{\sim}W(A)$ as shown in the last lemma. The claim follows.  \end{proof}
\medskip

\begin{cor} The overconvergent de Rham--Witt complex is continuous on the category of finite $k$-algebras.\end{cor}
\begin{proof}: Let $\varinjlim A_i=A$ be a filtered direct system of finite $k$-algebras. The restriction of the natural isomorphism
$$\varinjlim W\Omega_{A_i}\rightarrow W\Omega_A$$
to the overconvergent subcomplexes $W^\dagger\Omega_{A_i}$ has image in the overconvergent subycomplex $W^\dagger\Omega_A$. In fact it is exactly the natural homomorphism 
$$\varinjlim W^\dagger\Omega_{A_i}\rightarrow W^\dagger\Omega_A$$
induced by functoriality from $\varinjlim A_i\xrightarrow{\sim}A$ and the diagram
$$\xymatrix{ \varinjlim W^\dagger\Omega_{A_i}\ar[r]\ar[d] & W^\dagger\Omega_A\ar[d]\\ \varinjlim W\Omega_{A_i}\ar[r] & W\Omega_A}$$
commutes.

The morphism on the overconvergent complex is injective as the original map on the whole de Rham--Witt complex is injective. To check surjectivity, assume that $\omega\in W^\dagger\Omega_A$ with radius $\varepsilon>0$. By the previous proposition there is a unique preimage $\widetilde{\omega}\in \varinjlim W\Omega_{A_i}$. In particular, for each $i$, there is a $\omega_i\in W\Omega_{A_i}$ which maps to $\omega$ and it has to be overconvergent. Let $\varepsilon_i>0$ be its radius of overconvergence. Up to choosing different presentations, we may assume without loss of generality that the radii $\varepsilon_i$ are bounded below by $\varepsilon$. Hence, $\widetilde{\omega}$ can be considered as an element of $\varinjlim W^\dagger\Omega_{A_i}$.  \end{proof}
\medskip

\subsubsection*{The transformation map}\label{TransMap}

In the last two subsections we have shown that the complex $W^\dagger\Omega$ on the big \'etale site of all schemes is an object of the category $\Fan_{\text{\'et}}$, continuous on finite $k$-algebras. One should note that this is sufficient our case, although it is a priori a restriction of Kerz' definition, because we only wish to apply our functors to such cases. In fact, the overconvergent complex as suggested in \cite{DavisLangerZink} is only defined for finite $k$-algebras.

Moreover, we have seen that $\overline{\Ksheaf}_n^M$ is for every $n$ a continuous object of in $\Fan^\infty_{\text{\'et}}$ and that there exists a continuous $\widehat{\Ksheaf}_n^M\in\Fan_{\text{\'et}}$ and a natural transformation $\overline{\Ksheaf}_n^M\rightarrow \widehat{\Ksheaf}_n^M$ satisfying a universal property. This comprises all ingredients needed to apply Theorem \ref{ImprEt}.

As mentioned earlier, there is for each $n$ a morphism of continuous \'etale sheaves
$$d\log^n:\overline{\Ksheaf}^M_n\rightarrow W^\dagger\Omega\left[n\right].$$
As a consequence of Theorem \ref{ImprEt} which is based on Kerz's result (cf. Corollary \ref{KtoFunctor}) we obtain a unique natural transformation of \'etale sheaves
\begin{equation}\widehat{d\log}^n:\widehat{\Ksheaf}^M_n\rightarrow W^\dagger\Omega\left[n\right].\end{equation}
For simplicity, we will use the notation
\begin{equation}\label{dlog}d\log^n:\Ksheaf^M_n\rightarrow W^\dagger\Omega\left[n\right]\end{equation}
for the general case, where $\Ksheaf^M_n$ is the sheaf $\overline{\Ksheaf}^M_n$ in the infinite residue field case and $\widehat{\Ksheaf}^M_n$ in the finite residue field case.

\section{Chern classes with coefficients in the overconvergent de Rham-Witt complex}\label{Overconvergent}

In this section we introduce overconvergent Chern classes and look into some of their properties. 

\subsection{Definition and construction}

Let $X/k$ be a smooth variety. The map $d\log^n:\Ksheaf_n^M\rightarrow W^\dagger\Omega[n]$ defined in the previous section induces a morphism on cohomology
\begin{equation}\label{KOvercon}\h^m(X,\Ksheaf_i^M)\rightarrow\HH^{m+i}(X,W^\dagger\Omega_X),\end{equation}
which by abuse of notation we denote as well by $d\log$. 

It follows that the Chern classes $c_{ij}^M:K_j(X)\rightarrow\h^{i-j}(X,\Ksheaf_i^M)$ from Theorem \ref{TheoKsheafChern} induce Chern classes with coefficients in the overconvergent complex. 

\begin{theo} Let $X$ be a smooth scheme over $k$. There is a theory of Chern classes for vector bundles and higher algebraic $K$-theory of regular varieties over $k$, with values with coefficients in the overconvergent de Rham-Witt complex:
$$c_{ij}^{\text{sc}}:K_j(X)\rightarrow\HH^{2i-j}(X,W^\dagger\Omega_X).$$
\end{theo}

\begin{rem} Note that the maps $d\log^n$ respect the descending filtration of the de Rham-Witt complex by the differential graded ideals. Thus we obtain in fact Chern classes into the bi-graded cohomology groups
$$c_{ij}^{\text{sc}}:K_j(X)\rightarrow\HH^{2i-j}(X,W^\dagger\Omega^{\geqslant i}_X).$$
\end{rem}

\subsection{Comparison to crystalline Chern classes}\label{Sec2.2}

In this section assume that $X/k$ is proper in addition to being smooth. 

By construction the morphism $d\log^i:\Ksheaf^M_i\rightarrow W^\dagger\Omega\left[i\right]$ factors for each $i$ through $W^\dagger\Omega_{\log}^i=W\Omega_{\log}^i$. (For said equality see Section \ref{SubsecShortExactSequence}.) In the same manner as above, one may define logarithmic Chern classes
$$c_{ij}^{\log}:K_j(X)\rightarrow \HH^{i-j}(X,W\Omega^i_{\log})$$
which factor the overconvergent Chern classes. Thus it is natural to ask how these Chern classes compare to the logarithmic Chern classes with finite coefficients of Gros in \cite{Gros}
$$c_{ij}^{\text{Gros}}: K_j(X)\rightarrow\HH^{i-j}(X,W_n\Omega^i_{\log}).$$

\begin{prop} The diagram 
$$\xymatrix{& \h^{i-j}\left(X,W\Omega_{\log}^i\right)\ar[dd]\\
K_j(X)\ar[ur]^{c_{ij}^{\log}}\ar[dr]_{c_{ij}^{\text{Gros}}} &\\
& \h^{i-j}\left(X,W_n\Omega_{\log}^i\right)}$$
commutes.\end{prop}
\begin{proof}: Recall first that Gros proves a projective bundle formula for finite length log-differentials which allows him to build logarithmic Chern classes, which are then are uniquely defined. We will take advantage of this fact to show that Gros' Chern classes factor through the Milnor Chern classes.

The differential logarithm map $d\log:\OO_X^\ast\rightarrow W_n\Omega^1_{X,\log}$ as described by Gros, induces a map on the Milnor $K$-sheaf
$$d\log^i:\Ksheaf^M_i\rightarrow W_n\Omega^i_{X,\log}$$
in the same way as above and we have a commutative diagram of sheaves
\begin{equation}\label{diagKOmega}\xymatrix{& W\Omega^i_{\log}\ar[dd]^{\otimes\ZZ/p^n\ZZ}\\ \Ksheaf^M_i\ar[ur]\ar[dr] & \\ & W_n\Omega^i_{\log}}\end{equation}
It induces also morphisms of cohomology groups
$$\h^m(X,\Ksheaf^M_i)\rightarrow\h^m(X,W_n\Omega^i_{\log})$$
which is in particular compatible with multiplication in the cohomology ring. 

The fact that in both cases --- for the Milnor $K$-sheaf and for the logarithmic de Rham-Witt complex with finite coefficients --- we have a projective bundle formula, limits the comparison problem to the first Chern class. The diagram
$$\xymatrix{&\h^1(X,\Ksheaf^M_1)\ar[dd]^{d\log}\\ \h^1(X,\OO_X^\ast)\ar[ur]^{\id}\ar[dr]_{d\log}&\\ & \h^1(X,W_n\Omega^1_{\log})}$$
is commutative per definitionem and the upper morphism defines $c_1^M$ whereas the lower one defines $c_1^{\text{Gros}}$ (cf. \cite[Section I.2]{Gros}). With the respective projective bundle formulas and by using (\ref{diagKOmega}) this extends to a commutative diagram
$$\xymatrix{&\h^{i-j}(X,\Ksheaf^M_i)\ar[dd]\ar[r]^{d\log} & \h^{i-j}(X,W\Omega^i_{\log})\ar[ddl]^{\otimes\ZZ/p^n\ZZ}\\ K_j(X) \ar[ur]^{c_{ij}^M}\ar[dr]_{c_{ij}^{\text{Gros}}} & &  \\ & \h^{i-j}(X,W_n\Omega^i_{\log})}$$
and this is exactly what we claimed. \end{proof}

Because of properness of $X$ as assumed in the beginning of this section, we may compare the overconvergent Chern classes to crystalline Chern classes. From the discussion above and Gros' results we see that the diagram
$$\xymatrix{\OO_X\otimes\cdots\otimes\OO_X \ar[r]\ar[rrd] & \Ksheaf_i^M\ar[r] & W\Omega_{X,\log}^i\ar[d]^{\otimes\ZZ\slash p^n\ZZ}\ar[r] & W^\dagger\Omega_X\left[i\right]\ar[d]^{\otimes\ZZ\slash p^n\ZZ}\\ & & W_n\Omega_{X,\log}^i\ar[r] & W_n\Omega_X\left[i\right]}$$
commutes. Since Gros shows in \cite[Section 2.1]{Gros} that the logarithmic Chern classes he defines factor the crystalline Chern classes with $\mod p^n$ coefficients, the overconvergent Chern classes do the same, and one obtains a commutative diagram
$$\xymatrix{ & \HH^{2i-j}\left(X,W^\dagger\Omega\right)\ar[dd]\\
K_j(X)\ar[ur]^{c_{ij}^{\text{sc}}}\ar[dr]_{c_{ij}^{\text{cris},n}} &\\
& \HH^{2i-j}\left(X,W_n\Omega\right),}$$
that compares the overconvergent classes with finite level crystalline Chern classes. Taking limits we obtain
$$\xymatrix{ & \HH^{2i-j}\left(X,W^\dagger\Omega\right)\ar[dd]\\
K_j(X)\ar[ur]^{c_{ij}^{\text{sc}}}\ar[dr]_{c_{ij}^{\text{cris}}} &\\
& \h_{\text{cris}}^{2i-j}\left(X/W\right).}$$

\subsection{Overconvergent Chern classes and the $\gamma$-filtration}\label{AppendixGamma}

It is well known that Quillen's $K$-theory groups have a $\lambda$-structure -- more precisely for a given scheme $X$ $K_0(X)$ is a $\lambda$-ring and by Soul\'e the groups $K_m(X)$ can be equipped with a $K_0(X)$-$\lambda$ algebra structure. The operations $\gamma^k$ are then defined as shift $\gamma^k(x)=\lambda^k(x+k-1)$. For a noetherian, regular, connected scheme, we have the following $\gamma$-filtration compatible with products:

\begin{eqnarray*}F_\gamma^0 K(X) &=& K(X)\\
F_\gamma^j K(X) &=& \langle\gamma^{i_1}(x_1)\cdots\gamma^{i_n}(x_n)\,\big|\,\varepsilon(x_1)=\cdots=\varepsilon(x_n)=0,\, i_1+\cdots+i_n\geqslant j\rangle\end{eqnarray*}
We denote by $gr^i_\gamma K(X)=F^i_\gamma K(X)\slash F^{i+1}_\gamma K(X)$ the corresponding grading. 

\subsubsection*{Milnor Chern classes and the $\gamma$-filtration}

Let $X$ be smooth over $k$, no restrictions on the residue fields. Considering the fact that the overconvergent Chern classes are defined via Milnor $K$-theory, we may study the behaviour of the classes $c_{ij}^M:K_j(X)\rightarrow \h^{i-j}(X,\Ksheaf^M_i)$ on the filtration. 

As mentioned in Section \ref{Gillet} we know from \cite[Lemma 2.26]{Gillet} that the $c_{ij}^M$ for $j>0$ are group homomorphisms, which follows from the Whitney Sum Formula. In order to study how the Chern classes act on the $\gamma$-filtration we take a look at the product structure on $K$-theory. The multiplication as described by Loday is induced by a map
$$\mu_0: B_\cdot\GL(\OO_X)^+\times B_\cdot\GL(\OO_X)^+\rightarrow B_\cdot\GL(\OO_X)^+.$$
Arguing as in \cite[Lemma 2.32]{Gillet} we see that there is a commutative diagram
\begin{equation}\label{DiagLoday}\xymatrix{ B_\cdot\GL(\OO_X)^+\wedge B_\cdot\GL(\OO_X)^+\ar[r]^{\qquad\mu_0}\ar[d]^{C^M_\cdot\wedge C^M_\cdot} &  B_\cdot\GL(\OO_X)^+\ar[d]^{C^M_\cdot}\\
\prod_{i\in\NN}\Ksheaf(di,\Ksheaf^M_i)\wedge\prod_{i\in\NN}\Ksheaf(di,\Ksheaf^M_i)\ar[r]^{\qquad\qquad\ast} & \prod_{i\in\NN}\Ksheaf(di,\Ksheaf^M_i)}\end{equation}
where $\ast$ is Grothendieck's multiplication \cite{Grothendieck2} and $C_\cdot^M$ is the total Chern class. 

\begin{lem} If $\alpha\in K_l(X)$ and $\alpha'\in K_q(X)$ then
$$c_{ij}^M(\alpha\alpha')=-\sum_{r+s=i}\frac{(i-1)!}{(r-1)!(s-1)!}c_{rl}^M(\alpha)c_{sq}^M(\alpha'),$$
where $l+q=j$.\end{lem}
\begin{proof}: By property (\ref{TCCTensor}) of Definition \ref{TCC} we know that for the tensor product of two representations
$$\widetilde{C}^M_\cdot(\rho_1\otimes\rho_2)=\widetilde{C}^M_\cdot(\rho_1)\ast C_\cdot(\rho_2)$$
where $\widetilde{C}^M_\cdot$ is the total augmented Chern class, and the product is as above in the diagram the Grothendieck multiplication which after Shekhtman (see \cite[Section 2]{Niziol}) is given by universal polynomials
$$(\sum_{i\geqslant 1}x_i)\ast(\sum_{j\geqslant 1}y_i)=\sum_{l\geqslant 0}P_l(x_1,\ldots,x_l,y_1,\ldots,y_l)$$
with $P_l(x_1,\ldots,y_l))\sum_{r+s=l}a_{rs}x_r y_s+Z_r(x)T_s(y)$. Here 
$$a_{rs}=-\frac{(l-1)!)}{(r-1)!(s-1)!}$$
and $Z_r$ and $T_s$ are polynomials of weight $r$ and $s$ respectively and for $r+s=l$ at least one of them is decomposable. Explicitly, this means for the Loday multiplication $\mu_0$
$$\mu_0^\ast C^M_\cdot=\sum_{l\geqslant 0}\left(\sum_{r+s=l}\left(a_{rs}p_1^\ast C^M_r\cdot p_2^\ast C^M_s+ Z_r(p_1^\ast C^M_\cdot)T_s(p_2^\ast C^M_\cdot)\right)\right)$$
where $p_i:\GL(\OO_X)\times\GL(\OO_X)\rightarrow\GL(\OO_X)$ are the natural projections. Thus we have to show that the terms $Z_r(p_1^\ast C^M_\cdot)T_s(p_2^\ast C^M_\cdot)$ disappear when evaluated on the corresponding $K$-theory classes. The argumentation is analogue to the proof of \cite[Lemma 2.25]{Gillet}. An element $\eta\in K_j(X)$, $j\geqslant 1$ is represented by a map $\eta:\mathscr{S}^j_X\rightarrow B_\cdot\GL(\OO_X)^+$ in the homotopy category where $\mathscr{S}^j_X$ is the the simplicial version of the $j$-sphere. For any $a,b\in\NN$ the class $(C^M_a\cdot C^M_b)(\eta)$ is represented by the commutative diagram
$$\xymatrix{\mathscr{S}^j_X\ar[r]^{\Delta_{\mathscr{S}^j_X}}\ar[d]^\eta & \mathscr{S}^j_X\wedge\mathscr{S}^j_X\ar[d]^{\eta\wedge\eta}\\
 B_\cdot\GL(\OO_X)^+\ar[r]^{\Delta\qquad\qquad}\ar[d]^{C^M_a\cdot C^M_b}&  B_\cdot\GL(\OO_X)^+\wedge B_\cdot\GL(\OO_X)^+\ar[d]^{C_a^M\wedge C^M_b}\\
\Ksheaf(a+b,\Ksheaf^M_{a+b}) & \Ksheaf(a,\Ksheaf^M_a)\wedge\Ksheaf(b,\Ksheaf^M_b)\ar[l]_{\mu_{a,b}\quad}}$$
and as a consequence we have the equalities
$$(C_a^M\cdot C_b^M)\eta=\mu_{a,b}(C_a^M\wedge C_b^M)\Delta\eta=\mu_{a,b}(C_a^M\wedge C_b^M)(\eta\wedge\eta)\Delta_{\mathscr{S}^j_X}.$$
Since the map $\Delta_{\mathscr{S}^j_X}$ is null-homotopic, this composition of maps is null-homotopic as well. The decomposable part of the $Z_rT_s$ is made up by terms of this form, hence it disappears and with it the whole expression.  \end{proof}

\begin{lem} The integral Chern class maps $c_{ij}^M$ restrict to zero on $F^{i+1}_\gamma K_j(X)$ for $i\geqslant 1$.\end{lem}
\begin{proof}: In light of the previous formula and the Whitney sum formula, it suffices to show that $c_{ij}^M$ is trivial on elements of the form $\gamma_k(x)$ for $k\geqslant i+1$ and $x\in K_j(X)$. Recall from above that the operation $\gamma_k$ on $K_j(X)$ is defined to be the image of $\gamma_k(\id_N-N)$ under the natural map
$$r:R_{\ZZ}(\GL)\rightarrow[B_\cdot\GL(\OO_X)^+,B_\cdot \GL(\OO_X)^+]\rightarrow\Hom(K_j(X),K_j(X)).$$
As in \cite[Definition 2.27]{Gillet} we define the augmented cohomology group
$$\widetilde{\h}^\ast(X,\ZZ\times B_\cdot\GL(\OO_X)^+,\Ksheaf_i^M):=\h^0(X,\ZZ\times B_\cdot\GL(\OO_X)^+,\ZZ)\times\{1\}\times\h^i(X,\ZZ\times B_\cdot\GL(\OO_X)^+,\Ksheaf_i^M)$$
where $\h^i(X,\ZZ\times B_\cdot\GL(\OO_X)^+,\Ksheaf_i^M)=[\ZZ\times B_\cdot\GL(\OO_X)^+,\Ksheaf(i,\Ksheaf_i^M)]$. The ring $\widetilde{\h}^\ast(X,\ZZ\times B_\cdot\GL(\OO_X)^+,\Ksheaf_i^M)$ is a strict $\lambda$-ring \cite[2.27]{Gillet} and the augmented Chern class map
\begin{eqnarray*}\widetilde{C}^M:R_{\ZZ}(\GL)&\rightarrow&\widetilde{\h}^\ast(X,\ZZ\times B_\cdot\GL(\OO_X)^+,\Ksheaf_i^M)\\
\rho&\mapsto& \left(\rank(\rho),C_\cdot^M(\rho)\right)\end{eqnarray*}
is a $\lambda$-ring homomorphism. 

Now it suffices to show that the class of $c_{i,N}^M(\gamma_k(\id_N-N)))\in\h^i(X,\GL_N(\OO_X),\Ksheaf_i^M)$ is trivial for $k\geqslant i+1$. By the previous paragraph, the Chern polynomial $c_{t,N}^M$ is a $\lambda$-ring homomorphism as well, and hence commute with the $\gamma$-operation. With the usual formulae we get
\begin{eqnarray*}c_{t,N}^M(\gamma_k(\id_N-N)) &=& \gamma_k(c_{t,N}^M(\id_N-N))\\
&=& \gamma_k(1+c_{1,N}^M(\id_N-N)t+\cdots)\\
&=& 1+(-1)^{k-1}(k-1)!c_{k,N}^M(\id_N-N)t^k+\cdots\end{eqnarray*}
and therefore $c_{i,N}^M(\id_N-N))=0$ for $i<k$. \end{proof}

\begin{cor} If $\alpha\in F^j_\gamma K_l(X)$, $j\neq 0$ and $\alpha'\in F_\gamma^k K_q(X)$ then
$$c_{j+k,l+q}^M(\alpha\alpha')=-\frac{(j+k-1)!}{(j-1)!(k-1)!}c_{jl}^M(\alpha)c_{kq}^M(\alpha').$$
\end{cor}

\subsubsection*{Passage to overconvergent Chern classes}

Because the map 
$$\h^m(X,\Ksheaf^M_i)\rightarrow \HH^{m+i}(X,W^\dagger\Omega)$$
induced by the morphism of complexes $\Ksheaf^M_i\rightarrow W^\dagger\Omega[i]$ is a morphism of cohomology rings and therefore respects the respective operations, the results from the previous section carry over to the overconvergent Chern classes. This is summarised in the following proposition.

\begin{prop}\label{Prop534} Let $X/k$ be smooth. 
\begin{enumerate}\item If $\alpha\in K_l(X)$ and $\alpha'\in K_q(X)$ then
$$c_{ij}^{\text{sc}}(\alpha\alpha')=-\sum_{r+s=i}\frac{(i-1)!}{(r-1)!(s-1)!}c_{rl}^{\text{sc}}(\alpha)c_{sq}^{\text{sc}}(\alpha'),$$
where $l+q=j$.
\item The integral Chern class maps $c_{ij}^{\text{sc}}$ restrict to zero on $F^{i+1}_\gamma K_j(X)$ for $i\geqslant 1$.
\item If $\alpha\in F^j_\gamma K_l(X)$, $j\neq 0$ and $\alpha'\in F_\gamma^k K_q(X)$ then
$$c_{j+k,l+q}^{\text{sc}}(\alpha\alpha')=-\frac{(j+k-1)!}{(j-1)!(k-1)!}c_{jl}^{\text{sc}}(\alpha)c_{kq}^{\text{sc}}(\alpha').$$
\end{enumerate}
\end{prop}

\subsection{The action of $1-\f$ on the overconvergent de Rham-Witt complex}\label{SubsecShortExactSequence}

In this section, we want to adapt some short exact sequences from \cite{Illusie} to the overconvergent context.

We generalise the notion of basic Witt differentials to the case when $A$ is of the form $k\left[X_1,X_1^{-1},\ldots,X_d,X_d^{-1}\right]$. See the proof of Proposition 1.3 in \cite{DavisLangerZink}. A basic Witt differential $e\in W\Omega_A$ has one of the following shapes:
\begin{enumerate}\item $e$ is a classical basic Witt differential in variables $\left[X_1\right],\ldots,\left[X_d\right]$. 
\item Let $J\subset\{1,\ldots,d\}$ a subset and denote by $e(\xi,k,\PU,J)$ a basic classical Witt differential in $\{X_j\,|\,j\in J\}$.
\begin{enumerate}\item $e=e(\xi,k,\PU,J)\prod_{j\notin J}d\log\left[X_j\right]$.
\item $e=\prod_{j\notin J}\left[X_j\right]^{-r_j}e(\xi,k,\PU,J)$ for some $r_j\in\NN$.
\item $e=\prod_{j\notin J} {}^{\f^{s_j}}\!d\left[X_j\right]^{-l_j}e(\xi,k,\PU,J)$ for some $l_j\in\NN$, $p\nmid l_j$, $s_j\in\NN_0$.\end{enumerate}
\item $e= ^{\V^u}\!\left(\xi\prod_{j\notin J}\left[X_j\right]^{p^u{k_j}}\left[X\right]^{p^u{k_{I_0}}}\right)d {}^{\V^{u(I_1)}}\!\left[X\right]^{p^{u(I_1)}{k_{I_1}}}\cdots ^{\f^{-t(I_\ell)}}d\left[X\right]^{p^{t(I_\ell)}{k_{I_\ell}}}$. In particular, for each such $e$, there is a weight function on variables $\{X_j\,|j\in J\}$ with partition $\PU$, $u>0$, $k_{j\notin J}\in\ZZ_{<0}\left[\frac{1}{p}\right]$ and $u(k_j)\leqslant u=\max\{u(I_0),u(k_j)\}$.
\item $e=d e'$ where $e'$ as in (3).\end{enumerate}

\subsubsection*{The action of $\f$, $\V$ and $p$ on $W^\dagger\Omega$}

\begin{prop}\label{FrobAction} The action of $\f$ on the generalised basic Witt differentials are given as follows:
\begin{enumerate}\item If $e$ is a classical basic Witt differential in variables $\left[X_1\right],\ldots,\left[X_d\right]$, the action is given as in Proposition \ref{Prop4.2.1}. 
\item Let $J\subset\{1,\ldots,d\}$ a subset and denote by $e(\xi,k,\PU,J)$ a basic classical Witt differential in $\{X_j\,|\,j\in J\}$.
\begin{enumerate}\item If $e=e(\xi,k,\PU,J)\prod_{j\notin J}d\log\left[X_j\right]$, then 
$$^{\f}\!e=(^{\f}\!e(\xi,k,\PU,J))\prod_{j\notin J}d\log\left[X_j\right].$$
\item If $e=\prod_{j\notin J}\left[X_j\right]^{-r_j}e(\xi,k,\PU,J)$ for some $r_j\in\NN$, then 
$${}^{\f}\!e=\prod_{j\notin J}\left[X_j\right]^{-pr_j}({}^{\f}\!e(\xi,k,\PU,J)).$$
\item If $e=\prod_{j\notin J} {}^{\f^{s_j}}\!d\left[X_j\right]^{-l_j}e(\xi,k,\PU,J)$ for some $l_j\in\NN$, $p\nmid l_j$, $s_j\in\NN_0$ then
$$^{\f}\!e=\prod_{j\notin J} {}^{\f^{s_j+1}}\!d\left[X_j\right]^{-l_j}({}^{\f}\!e(\xi,k,\PU,J)).$$\end{enumerate}
\item If $e= {}^{\V^u}\!\left(\xi\prod_{j\notin J}\left[X_j\right]^{p^u{k_j}}\left[X\right]^{p^u{k_{I_0}}}\right)d {}^{\V^{u(I_1)}}\!\left[X\right]^{p^{u(I_1)}{k_{I_1}}}\cdots ^{{}\f^{-t(I_\ell)}}\!d\left[X\right]^{p^{t(I_\ell)}{k_{I_\ell}}}$, then
$${}^{\f}\!e={}^{\V^u}\!\left({}^{\f}\!\xi\prod_{j\notin J}\left[X_j\right]^{p^u{k'_j}}\left[X\right]^{p^u{k'_{I_0}}}\right)d {}^{\V^{u(I_1)}}\!\left[X\right]^{p^{u(I_1)}{k'_{I_1}}}\cdots {}^{\f^{-t(I_\ell)}}\!d\left[X\right]^{p^{t(I_\ell)}{k'_{I_\ell}}},$$
where $k'=pk$. 
\item If $e=d e'$ where $e'$ as in (3), the expression changes similar to the previous case, with the only difference that we get ${}^{\V^{-1}}\!\xi$ instead of ${}^{\f}\!\xi$.\end{enumerate}\end{prop}
\begin{proof}: This is a straight forward calculation, using the definition of Frobenius and Verschiebung.  \end{proof}

In particular, $\f$ has the same stabilizing properties on the types of generalised basic Witt differentials as mentioned at the end of Section \ref{SubsecBasic} with respect to the usual basic Witt differentials.

\begin{rem}\label{RemFrobAction}In this concrete case we can give a criterion when an element $\omega= \sum e(\xi,k,\PU)$ of the de Rham-Witt complex given as its decomposition in basic generalised Witt differentials is overconvergent based on the proof of Proposition 1.3 in \cite{DavisLangerZink}. Namely, $\omega$ is overconvergent, if there exist constants $C_1>0$ and $C_2\in\RR$ such that the basic Witt differentials $e$ appearing in the decomposition satisfy the following conditions:

\begin{itemize}\item If $e$ is of type (1) or of type (2.a), 
$$|k|\leqslant C_1\ord_p\xi_{k,\PU}+C_2.$$
\item If $e$ is of type (2.b),
$$|r|+|k|\leqslant C_1\ord_p\xi_{k,\PU}+C_2,$$
where $|r|=\sum r_j$.
\item If $e$ is of type (2.c),
$$|l\cdot p^s| + |k|\leqslant C_1\ord_p\xi_{k,\PU}+C_2,$$
where $|l\cdot p^s|=\sum l_j\cdot p^{s_j}$.
\item If is of type (3) or (4),
$$\sum|k_j|+\sum|k_{I_i}|\leqslant C_1\ord_p({}^{V^u}\!\xi)+C_2,$$
where $|k_j|=-k_j$ and $|k_{I_i}|=\sum_{m\in I_i}k_m$.\end{itemize}\end{rem}

The multiplication of an element $\alpha\in W(k)$ on a generalised basic Witt differential is particularly easy to define.

\begin{prop} The action of $\alpha\in W(k)$ on a generalised basic Witt differential is given by multiplying the coefficient $\xi$ with $\alpha$.\end{prop}
\begin{proof}: Analogues to the discussion in \cite[p.40]{LangerZink} we see that the coefficients in the generalised basic Witt differentials  are elements of $^{\V^{U(I)}}W(k)$ where $I$ is the partition of the support of the weight function in question and $u$ is defined as in Section \ref{SubsecBasic}. As for any $\alpha\in W(k)$ and $\xi\in ^{\V^{U(I)}}W(k)$ the product $\alpha\xi$ is again in $^{\V^{U(I)}}W(k)$, we see as in \cite[loc.sit.]{LangerZink} that multiplying by $\alpha$ a general basic Witt differential of one of the forms given at the begin of the section means to multiply the coefficient $\xi$ appearing there by $\alpha$.  \end{proof}

In particular, multiplication by an element in $W(k)$ respects the types of general basic Witt differentials, and this is the fact that we will use later for multiplication by $p$.

Due to the fact that the Verschiebung is only additive, but not a ring homomorphism, its action on the generalised basic Witt differentials is more complicated to describe. We have recalled the action on the usual basic Witt differentials in Proposition \ref{Prop4.2.1}. We note further, that 
$${}^{\V}\!d\log[X_i]=d{}^{\V}\!\log[X_i]$$
as the general formula ${}^{\V}\!(\omega_0d\omega_1\cdots d\omega_i)={}^{\V}\!\omega_0 d{}^{\V}\!\omega_1\cdots d{}^{\V}\!\omega_i$ holds. This formula also tells us, that for differentials of type (3), i.e. if $e= {}^{\V^u}\!\left(\xi\prod_{j\notin J}\left[X_j\right]^{p^u{k_j}}\left[X\right]^{p^u{k_{I_0}}}\right)d {}^{\V^{u(I_1)}}\!\left[X\right]^{p^{u(I_1)}{k_{I_1}}}\cdots {}^{\f^{-t(I_\ell)}}\!d\left[X\right]^{p^{t(I_\ell)}{k_{I_\ell}}}$, then 
$$^{\V}\!e={}^{\V^u+1}\!\left(\xi\prod_{j\notin J}\left[X_j\right]^{p^u{k_j}}\left[X\right]^{p^u{k_{I_0}}}\right)d {}^{\V^{u(I_1)+1}}\!\left[X\right]^{p^{u(I_1)}{k_{I_1}}}\cdots {}^{\f^{-t(I_\ell)-1}}\!d\left[X\right]^{p^{t(I_\ell)}{k_{I_\ell}}}.$$
Furthermore, it is possible to describe the action of $\V$ on elements of type (4) using the same formula by multiplying with the factor 1, which allows us to write
$$^{\V}\! d e={}^{V}\!1 d^{\V}\!e$$
which changes the coefficient to $p^{\V}\xi$ (see \cite[p. 41]{LangerZink} and we obtain
$$^{\V}\!e= d {}^{\V^u}\!\left(p{}^{\V}\!\xi\prod_{j\notin J}\left[X_j\right]^{p^u{k_j}}\left[X\right]^{p^u{k_{I_0}}}\right)d {}^{\V^{u(I_1)}+1}\!\left[X\right]^{p^{u(I_1)}{k_{I_1}}}\cdots {}^{\f^{-t(I_\ell)}-1}\!d\left[X\right]^{p^{t(I_\ell)}{k_{I_\ell}}}.$$
As for general basic differentials of type (2), this depends on the different cases and also on the form of the basic Witt differentials $e(\xi,k,\PU,J)$ involved.

\subsubsection*{The kernel and cokernel of $1-\f$ for $X=\GG_m^d$}

The logarithmic differentials are by definition of the de Rham-Witt complex fixed by the Frobenius endomorphism. One would like to prove that these are the only elements with this property. We start by making the assertion for the special case, when $X$ is a product of multiplicative groups, i.e. $X=\Spec A$ with $A=k\left[X_1,X_1^{-1},\ldots,X_d,X_d^{-1}\right]$. 

\begin{prop} The sections of $W^\dagger\Omega_{X,\log}$ on $X=\GG_m^d$ are exactly the Frobenius fixed elements of $W\Omega_X$.\end{prop}
\begin{proof}:  It is clear that $W^\dagger\Omega_{X,\log}\subset (W\Omega_X)^{\f-1}$. For the converse, let $\omega\in(W\Omega_X)^{\f-1}$, with decomposition in generalised basic Witt differentials $\omega=\sum e$. However, we know that the action of Frobenius preserves the types of basic Witt differentials and that the decomposition is unique. We want to use this to argue that it is enough to check the generalised basic Witt differentials.

With the assumption $\f \omega=\omega$ there are two cases to consider: finite and infinite sets of elements appearing in the sum decomposition of $\omega$ that form a (finite respectively infinite) ``cycle'' under the action of $\f$.\\
\textbf{The finite case}: assume that among the basic Witt differentials in the decomposition $\omega=\sum e$ there is a finite set $e_1,\ldots,e_n$ such that $\f\sum_{i=1}^n e_i=\sum_{i=1}^n e_i$ which means after possibly reordering
$$\f e_1=e_2\; ,\; \f e_2=e_3\; ,\; \ldots,\; \f e_n=e_1.$$
But according to Proposition \ref{FrobAction} this is impossible unless $n=1$.\\
\textbf{The infinite case}: assume that in the decomposition there is an infinite set $e_1,e_2,\ldots$ such that 
$$\f\sum_{i=1}^\infty e_i=\sum_{i=1}^\infty e_i$$
which means after reordering
$$\f e_1=e_2\; ,\; \f e_2=e_3\; ,\; \ldots$$
However this is not convergent $p$-adically and therefore not feasible.

Thus if the whole sum $\omega=\sum e$ is fixed under Frobenius, every basic differential appearing in this sum must be so.

According to Proposition \ref{FrobAction}, the sections fixed under Frobenius action are of the form (2.a) with $e(\xi,k,\PU,J)$ trivial. This shows that $(W\Omega_X)^{\f-1}\subset W^\dagger\Omega_{X,\log}$.  \end{proof}

It follows in particular that we have for all $i\in\NN$ locally for \'etale topology a commutative diagram where the rows are exact:
$$\xymatrix{0\ar[r] & W\Omega^i_{X,\log}\ar@{^{(}->}[r]\ar@{=}[d] & W\Omega^i_X\ar[r]^{\f-1} & W\Omega^i_X\ar[r] & 0\\
0\ar[r] & W^\dagger\Omega^i_{X,\log}\ar@{^{(}->}[r]\ar@{^{(}->}[ru] & W^\dagger\Omega^i_X\ar[r]^{\f-1}\ar@{_{(}->}[u] & W^\dagger\Omega^i_X\ar@{_{(}->}[u] & }$$

Earlier we mentioned the condition for an element $\omega\in W\Omega_{X/k}$ to be overconvergent. On the other hand, this shows that an element $\omega\in W\Omega_{X/k}$ is \textbf{not} overconvergent, if for all $C_1>0$ and $C_2\in \RR$ there is an elementary Witt differential $e$ in its decomposition violating one of the inequalities.  We will study the action of Frobenius on elementary Witt differentials subject to the inequalities indicating overconvergence or non-overconvergence.

As mentioned before, any Witt differential over $\GG_m^d$ can be written in a unique way as a sum of basic Witt differentials. The basic Witt differentials appearing in this sum are characterised by a coefficient, a weight function and a partition. The action of Frobenius was described in Proposition \ref{FrobAction}. In essence Frobenius action changes the weights and the coefficients of the basic differentials creating a new unique sum. One can see Frobenius as being injective on the \textbf{set} of basic Witt differentials. 

We want to argue that overconvergence is preserved by Frobenius if we modify one of the constants in an obvious way.

\begin{lem} Let $e$ be a basic Witt differential satisfying an inequality indicating overconvergence for constants $C_1$ and $C_2$, then so does $^{\f}\! e$ for constants $C_1$ and $pC_2$.\end{lem}
\begin{proof}: This is shown for one type at a time. 
\begin{itemize} \item If $e$ is of type (1) or of type (2.a), then the inequality depends only on the basic classical Witt differential appearing in the expression: $|k| \leqslant C_1\ord_p\xi_{k,\PU}+C_2$. The action of $\f$ on a differential of this type changes the weight $k$ to $pk$ and the coefficient $\xi_{k,\PU}$ to $^{\f}!\xi_{k,\PU}$ if $k$ is integral and $^{V^{-1}}!\xi_{k,\PU}$ if $k$ is fractional, the partition is essentially unchanged. The inequality has to be modified to
$$|pk|=p|k| \leqslant C_1\ord_p({}^{\f}\!\xi_{k,\PU})+pC_2.$$
\item If $e$ is of type (2.b), the crucial inequality is $|r|+|k|\leqslant C_1\ord_p\xi_{k,\PU}+C_2$, where $|r|=\sum r_j$. The action of $\f$ changes $r_j$ to $pr_j$ and therefore $|r|$ to $p|r|$, $k$ and $\xi$ change as in the previous case. As above, it becomes clear that the only modification to the constants has to be $pC_2$ instead of $C_2$.
\item If $e$ is of type (2.c), the same argument is valid with $|l|$ in the place of $|r|$.
\item If $e$ is of type (3) or (4), action of Frobenius means that the $k_I$'s appearing are multiplied by $p$ and the coefficient changes to $^{\f}\!\xi$ or $^{V^{-1}}\!\xi$. Again we see that the inequality still holds if we change $C_2$ to $pC_2$.
\end{itemize}
This shows the claim. \end{proof}

\begin{lem} Let $e$ be a basic Witt differential satisfying an inequality indicating non-overconvergence for constants $C_1$ and $C_2$, then so does $^{\f}\!e $ for the same constants.\end{lem}
\begin{proof}: This is essentially the same argument as before, with the difference, that this time we deal with strict inequalities in the other direction, so there is no need to increase the second constant.  \end{proof}

\begin{prop}\label{PropC14} The map $1-\f$ over $X=\GG_m^d$ is surjective for \'etale topology.\end{prop}
\begin{proof}: Let $\omega\in W^\dagger\Omega_{X/k}$. We have seen that up to \'etale localisation, there is $\eta\in W\Omega_{X/k}$ such that $\omega=(1-\f)\eta$. We have to show that $\eta$ is in fact overconvergent. 

Write $\eta=\sum_{k,\PU,J} e(\xi,k,\PU,J)$ as a unique sum of elementary Witt differentials and assume that it is not overconvergent. Then for all $C_1>0$ and $C_2\in\RR$ there is an element $e$ appearing in the sum that violates the inequalities for overconvergence, or in other words satisfies the strict inequalities for non-overconvergence, and so do the elements $^{\f^i}\!e$, $i\in\NN_0$ for the same constants $C_1$ and $C_2$. 

Since $\omega=(1-\f)\sum_{k,\PU,J} e(\xi,k,\PU,J)$ is overconvergent there must be $C_1, C_2$ for which the corresponding elements that violate the overconvergence inequality in the original sum cancel out after applying $1-\f$. Let $e$ be one of these elements. 

The image of $e$ is $e-{}^{\f}\!e$. Due to the nature of the basic Witt differentials and the way Frobenius acts on the different types as pointed out in Remark \ref{RemFrobAction}, it is clear that $e$ (and similarly $^{\f}\!e$) either remains and appears as a basic Witt differential in the unique decomposition of $\omega$ or is cancelled out by some basic Witt differential $^{\f}\!e'$ where $e'$ is another basic Witt differential of the decomposition of $\eta$ (similarly $^{\f}\!e$ appears or is cancelled out by a basic Witt differential $e''$ which appears in the decomposition of $\eta$). 

Since we assumed that $e$ be cancelled out after applying $1-\f$ the same must hold true for $^{\f}\!e$ which is subject to the same inequality. Hence $e''= {}^{\f}\!e$ has to appear in the unique sum of $\eta$. By induction the basic Witt differentials $^{\f^i}\!e, i\in\NN_0$ all appear in the unique sum of $\eta$. But a sum containing all of these elements cannot be convergent in the sense of Section \ref{SubsecBasic}. This is a contradiction to the uniqueness of the sum and therefore $\eta$ must be overconvergent to begin with.  \end{proof}

\begin{cor}\label{CorSimplified} The same is true for $X=\Spec A[Y,Y^{-1}]$ where $A=k[X_1,\ldots,X_d]$.\end{cor}
\begin{proof}: This is just a simplified version of the previous assertion. \end{proof}

We have now established the following exact sequence for $X=\GG_m^d$ for \'etale topology
$$0\rightarrow W^\dagger\Omega_{X,\log}^i\rightarrow W^\dagger\Omega_X^i\xrightarrow{\f-1}W^\dagger\Omega_X^i\rightarrow 0.$$
We want to extend this result to general smooth schemes over $k$. 

\subsubsection*{The map $1-\f$ over a smooth $k$-scheme}

First we note, that we can reduce the general case to the case of a localised polynomial algebra. By a result of Kedlaya \cite{Kedlaya3} any smooth variety has a cover by standard \'etale affines as defined in \cite{Raynaud}. What is more, this cover can be chosen in a way that any finite intersection is again standard \'etale affine (see \cite[Proposition 4.3.1]{Davis}). Let $A=k[X_1,\ldots,X_d]$ and $f\in A$. In the proof of Theorem 1.8 in \cite{DavisLangerZink} the authors argue that it suffices to consider finite \'etale monogenic algebras over rings of the form $A_f$. In \cite[Proposition 1.9]{DavisLangerZink} they reduce this further by stating

\begin{prop}\label{PropC12} Let $B$ a finite \'etale and monogenic $C$-algebra, where $C$ is smooth over a perfect field of char $p>0$. Let $B=C[X]/(f(X))$ for a monic polynomial $f(X)$ of degree $m=\left[B:C\right]$ such that $f'(X)$ is invertible in $B$. Let $[x]$ be the Teichm\"uller of the element $X\mod f(X)$ in $W(B)$. Then we have for each $d\geqslant 0$ a direct sum decomposition of $W^\dagger(C)$-modules
$$W^\dagger\Omega^d_{B/k}=W^\dagger\Omega_{C/k}^d\oplus W^\dagger\Omega_{C/k}^d[x]\oplus\cdots\oplus W^\dagger\Omega_{C/k}^d[x]^{m-1}.$$
\end{prop}

Finally they prove that the overconvergent de Rham-Witt complex over a smooth $k$-scheme $X$ is a complex of \'etale (and Zariski) sheaves on $X$. Thus we see that for our purposes as we seek a local result it is enough to consider the (overconvergent) de Rham-Witt complex over a localised polynomial algebra of the form $A_f$.

Now we proceed to calculate kernel and cokernel of the map $1-\f$. 

\begin{lem}\label{LemC22} Let $X=\Spec A_f$. The map $\f-1$ on $W^\dagger\Omega_{X/k}$ is surjective for \'etale topology.\end{lem}
\begin{proof}: Consider the $k$-algebra $A[Y,Y^{-1}]$. There is a canonical surjection
\begin{eqnarray*} A[Y,Y^{-1}] &\rightarrow& A_f\\ Y &\mapsto& f.\end{eqnarray*}
This induces by functoriality a surjection of the associated de Rham-Witt complexes $W\Omega_{A[Y,Y^{-1}]/k}\rightarrow W\Omega_{A_f/k}$. For quotients of polynomial algebras an element of the corresponding de Rham-Witt complex is said to be overconvergent if there exist a lift of this element to the polynomial algebra which is overconvergent. Therefore we have in fact a surjection of overconvergent de Rham-Witt complexes
$$W^\dagger\Omega_{A[Y,Y^{-1}]/k}\rightarrow W^\dagger\Omega_{A_f/k}.$$
Moreover, there is a commutative diagram
$$\xymatrix{W^\dagger\Omega_{A[Y,Y^{-1}]/k} \ar[r]\ar[d]_{\f-1} & W^\dagger\Omega_{A_f/k}\ar[d]^{\f-1}\\W^\dagger\Omega_{A[Y,Y^{-1}]/k}\ar[r] & W^\dagger\Omega_{A_f/k}}$$
By Corollary \ref{CorSimplified} the vertical map on the left is surjective and thus the same holds true for the one on the right.  \end{proof}

\begin{lem} The kernel of $\f-1$ on $W^\dagger\Omega_{X/k}$ for $X=\Spec A_f$ is $W^\dagger\Omega_{X/k,\log}$.\end{lem}
\begin{proof}: By definition of the complex $W^\dagger\Omega_{X/k,\log}$ and the discussion above it is contained in the kernel of $\f-1$. On the other hand, we have seen, that $W^\dagger\Omega_{X/k,\log}=W\Omega_{X/k,\log}$ and this is the kernel of $\f-1$ on $W\Omega_{X/k}$ without the overconvergence condition. Thence the restriction of $\f-1$ to the overconvergent subcomplex $W^\dagger\Omega_{X/k}$ must have the same kernel.  \end{proof}

Combining the arguments of the last two sections yields

\begin{cor}\label{CorC25} Let $X$ be a smooth scheme over a perfect field $k$ of characteristic $p>0$. Then for all $i\in\NN$ there is locally for \'etale topology a short exact sequence
$$0\rightarrow W^\dagger\Omega^i_{X/f,\log}\rightarrow W^\dagger\Omega^i_{X/k}\xrightarrow{\f-1} W^\dagger\Omega^i_{X/k}\rightarrow 0.$$
\end{cor}

\subsubsection*{Chern classes into the Frobenius fixed part}

Unfortunately the Frobenius morphism and therefore also the morphism $1-\f$ is only a ring homomorphism and not a morphism of complexes. The reason for this is, that $\f$ does not commute with the differential, in fact the formula
\begin{equation}\label{FDcommuting}d\f=p\f d\end{equation}
holds. Thus some modifications are required which are inspired by \cite[Corollaire I.3.29 and Th\'eor\`eme II.5.5]{Illusie}. 

\begin{defi} For each $m\geqslant 1$ we define an endomorphism of complexes
$$\f_m:W\Omega^{\geqslant m}\rightarrow W\Omega^{\geqslant m},$$
where we use the na\"ive truncation, which is given in degree $i\geqslant m$ by $p^{i-m}\f$. \end{defi}
Looking at the commuting formula (\ref{FDcommuting}) of $d$ and $\f$, it is clear now that this definition gives indeed a morphism of complexes, and by extension the same holds true for $1-\f_m$. 

Illusie shows in \cite[Lemme I.3.30]{Illusie} that for all $r\geqslant 1$ and all $i\geqslant 0$, the morphism $1-p^r\f$ is an automorphism of the pro-object $W_{_\bullet}\Omega^i_X$, hence (for example using the Mittag-Leffler condition) the induced map on $W\Omega^i_X$ is also an automorphism. 

We consider now the restriction of theses morphisms to the overconvergent subobjects. We have already seen, that multiplication by $p$ and the Frobenius $\f$ map overconvergent elements to overconvergent elements. Thus for $m\geqslant 1$ there is an endomorphism of complexes
$$1-\f_m:W^\dagger\Omega^{\geqslant m}\rightarrow W^\dagger\Omega^{\geqslant m}$$
and for $r\geqslant 1$ and $i\geqslant 0$ the morphism $1-p^r\f:W^\dagger\Omega^i_X\rightarrow W^\dagger\Omega^i_X$ as restriction from the usual de Rham-Witt complex is injective.

Unfortunately, the argument from Proposition \ref{PropC14}, where we show that $1-\f$ is surjective, does not work in this case, as subsequent multiplication of a generalised basic Witt differential by $p^r$ for a fixed $r>1$ creates and overconvergent sequence. Together with the short exact sequence from Corollary \ref{CorC25} we obtain for a fixed $m\geqslant 1$ the following exact sequence of complexes:
\begin{equation}\label{ExactComplexSequence} 0\rightarrow W^\dagger\Omega^m_{X/k,\log}[-m]\rightarrow W^\dagger\Omega^{\geqslant m}_{X/k}\xrightarrow{1-\f_m} W^\dagger\Omega^{\geqslant m}_{X/k}.\end{equation}
However, this is enough for our purposes. 

The exact sequence (\ref{ExactComplexSequence}) induces an exact sequence on cohomology.

\begin{prop} For $X/k$ smooth and $m\in\NN_0$, $i\in\NN$ there is an exact sequence
$$0\rightarrow \h^m(X,W^\dagger\Omega^i_{\log})\rightarrow \HH^{m+i}(X,W^\dagger\Omega^{\geqslant i})\xrightarrow{1-\f_i}\HH^{m+i}(X,W^\dagger\Omega^{\geqslant i}).$$
\end{prop}
\begin{proof}: We start with the exact sequence (\ref{ExactComplexSequence}) and replace the rightmost object by the image of $1-\f_m$, which makes it into a short exact sequence. In the associated long exact sequence of cohomology the connecting morphisms are obviously trivial and after going back to the original complexes we obtain the above sequence for each $i$ and $m$.  \end{proof}

In particular, we obtain the following result.
\begin{cor} Let $X/k$ be smooth and $m\in\NN_0$, $i\in\NN$. Then we have the identity
$$\HH^{m+i}(X,W^\dagger\Omega^{\geqslant i})^{1-\f_i}=\h^m(X,W^\dagger\Omega^i_{\log}).$$
\end{cor}
The submodule $\HH^{m+i}(X,W^\dagger\Omega^{\geqslant i})^{1-\f_i}\subset \HH^{m+i}(X,W^\dagger\Omega^{\geqslant i})$ can be thought of as the Frobenius eigen module of eigenvalue $\frac{1}{p^m}$. 

Now recall that by construction the overconvergent Chern classes factor through the logarithmic differentials. Therefore the stated identity entails the subsequent corollary.

\begin{cor} Let $X/k$ be smooth. Then the overconvergent Chern classes constructed earlier can be written as
$$c_{ij}^{\text{sc}}: K_j(X)\rightarrow \HH^{2i-j}(X,W^\dagger\Omega^{\geqslant i})^{1-\f_i}.$$
\end{cor}

Taking into account the $\gamma$-filtration, especially Proposition \ref{Prop534} (2) yields Chern classes on the $\gamma$-graded pieces of the algebraic $K$-groups. 
\begin{cor} Let $X/k$ be smooth. There are overconvergent Chern classes
$$c_{ij}^{\text{sc}}:\gr_\gamma^i K_j(X)\rightarrow \HH^{2i-j}(X,W^\dagger\Omega^{\geqslant i})^{1-\f_i}.$$
\end{cor}

\section{Comparison of overconvergent and rigid Chern classes}\label{Comparison}

The purpose of this section is to show that in the case of a smooth and quasi-projective variety the overconvergent Chern classes from the previous section are compatible with the rigid Chern classes defined by Petrequin in \cite{Petrequin}.

\subsection{Rigid Chern classes}\label{RigChernClasses}

Let $X$ be a proper variety over $k$ and $\ver$ a discrete valuation ring with residue field $k$. We choose a closed immersion $X\hookrightarrow\yer$, where $\yer$ is a formal scheme over $\Spf(\ver)$ smooth in a neighbourhood of $X$. 
Petrequin defines in \cite[Section 1.2]{Petrequin} the cohomology class asociated to a good pseudo-divisor in the sense of Fulton using \v{C}ech cohomology. Functoriality allows this to be generalised for open varieties. 

Consider  a (smooth) $k$-variety $X$ and a vector bundle $\edg$ of rank $r$ over $X$. We denote by $\pi:\PP=\PP_{\edg}\rightarrow X$ the associated projective bundle. Let 
$$\xi=c_1^{\text{rig}}(\OO_{\PP}(1))\in\h^2_{\text{rig}}(\PP\slash K)$$
be the class of the good pseudo-divisor $(\OO_{\PP}(1),X,-)$. According to Petrequin this can be calculated by the \v{C}ech cocycle 
$$\left(\frac{du}{u}\right)\in Z^2(\mathfrak{U}_K,\Omega^\ast_{]\PP[}),$$
where $\mathfrak{U}$ is a covering trivialising $\OO_{\PP}(1)$ over $\ver$. By \cite[Corollary 4.4]{Petrequin} there is a projective bundle formula
$$\h^n_{\text{rig}}(\PP\slash K)\cong \bigoplus_{i=0}^{r-1}\h^{n-2i}_{\text{rig}}(X\slash K)\xi^i.$$
As usual, this produces welldefined Chern classes $c_i^{\text{rig}}(\edg)\in\h^{2i}_{\text{rig}}(X\slash K)$.  and then induces a theory of Chern classes for higher algebraic $K$-theory with coefficients in rigid cohomology
$$c_{ij}^{\text{rig}}:K_j(X)\rightarrow \h^{2i-j}_{\text{rig}}(X\slash K).$$

\begin{prop}\label{RigChernFac} Let $X/k$ be a smooth variety. The rigid Chern classes defined by Petrequin factor through Milnor $K$-theory via a morphism
\begin{equation}\label{CompMilnorRig}\h^i(X,\Ksheaf^M_m)\rightarrow \h^{i+m}_{\text{rig}}(X/K).\end{equation}
\end{prop}
\begin{proof}: We start with the case of infinite residue fields. If $X$ is not proper let $j:X\hookrightarrow \overline{X}$ be a suitable compactification, which exists by \cite[Lemme 3.19]{Petrequin}, otherwise take $X=\overline{X}$. Furthermore, let $\overline{X}\hookrightarrow\yer$ be a closed immersion into a formal scheme over $\Spf(\ver)$ as before, and $]X[$ and $]\overline{X}[$ the tubes of $X$ and $\overline{X}$ respectively in the generic fibre of $\yer$. 

For a local section $x$ of $\OO_X^\ast$ choose a lift $\tilde{x}$ over $\yer$. This can be seen as a rigid analytic function on $]X[$ and $]\overline{X}[$ that restricts to an invertible element on $X$, which is therefore itself invertible (as rigid analytic function). We thus set
$$\mu=\frac{d\tilde{x}}{\tilde{x}}$$
and thereby define a local section of $\Omega^1_{]\overline{X}[}$ whose cocycle class is independent of the choice of lift. In the same manner, it is possible to assign to a local section $x_1\otimes\cdots\otimes x_i\in \OO_X^\ast\otimes\cdots\otimes\OO_X^\ast$ a section
$$\mu_1\cdots\mu_i=\frac{d\tilde{x}_1}{\tilde{x}_1}\cdots\frac{d\tilde{x}_i}{\tilde{x}_i}$$
of $\Omega^i_{]\overline{X}[}$. It is clear that 
$$\frac{d\tilde{x}}{\tilde{x}}\frac{d(\tilde{1}-\tilde{x})}{(\tilde{1}-\tilde{x})}=0$$
and therefore, as $\tilde{1}-\tilde{x}$ is a lift of $1-x$, the classes of the symbols $\mu_1\cdots\mu_i$ satisfy the Steinberg relation. Consequently, there are induced morphisms of cohomology groups
\begin{equation}\label{MorphMRig}\h^m(X,\overline{\Ksheaf}^M_i)\rightarrow\HH^{i+m}\left(]\overline{X}[,j^\dagger\Omega_{]\overline{X}[}\right)=\h^{i+m}_{\text{rig}}(X/K),\end{equation}
and it is clear that it respects the multiplicative structure of the cohomology rings. 

To see that this factors the rigid Chern classes, consider a vector bundle $\pi:\edg\rightarrow X$ of constant rank $n$, and let $\PP=\PP(\edg)$ be the associated projective bundle. As we mentioned there is a cover $\mathfrak{U}=(\mathscr{U}_i)$ of $\yer$ such that the induced cover $\mathfrak{U}_X=(U_i)$ of $X$ trivialises the line bundle $\OO_{\PP}(1)$. To this trivialisation we can associate in the classical manner a \v{C}ech cocycle
$$(u)\in Z^1(\mathfrak{U}_X,\OO^\ast_X)$$
as $u_{ij}$ on $U_i\cap U_j$, which calculates the first Chern class of $\OO_{\PP}(1)$ in $\h^1(\PP,\overline{\Ksheaf}_1^M)$. On the other hand, Petrequin shows that the first Chern class of $\OO_{\PP}(1)$ in $\h^2_{\text{rig}}(\PP/K)$ is given by the \v{C}ech cocycle 
$$\left(\frac{d\tilde{u}}{\tilde{u}}\right)\in Z^2(\mathfrak{U}_K,\Omega_{]\PP[}),$$
and its class agrees with the image of the class of $(u)$ under the morphism (\ref{MorphMRig}). In both cases the Chern classes $c_i^{\text{rig}}(\edg)\in\h^{2i}_{\text{rig}}(X/K)$ and $c^M_{i}(\edg)\in\h^i(X,\overline{\Ksheaf}^M_i)$ are uniquely defined via a projective bundle formula using the same relations namely
\begin{eqnarray*} c_0(\edg) &=& 0,\\
c_{i>r}(\edg) &=& 0,\\
\sum_{i=0}^r c_i(\edg)c_1\left(\OO_{\PP}(1)\right)^{r-i} &=& 0.\end{eqnarray*}
Indeed, we have a commutative diagram of the form
$$\xymatrix{\h^j(\PP,\K_j^M)\ar[r] & \h^j_{\text{rig}}(\PP/K)\\ 
\bigoplus_{i=0}^{n-1}\h^{j-i}(X,\overline{\Ksheaf}_{j-i}^M)\cdot c_1^M(\OO_{\PP}(1))^i\ar[r]\ar[u]^{\sim} & \bigoplus_{i=0}^{n-1}\h^{i-2j}_{\text{rig}}(X/K)\cdot c_1^{\text{rig}}(\OO_{\PP}(1))^i\ar[u]_{\sim}}.$$
The fact that the morphism (\ref{CompMilnorRig}) is compatible with multiplication shows that the rigid Chern classes $c_i^{\text{rig}}$ factor through Milnor $K$-theory sheaves  The same then holds for the higher Chern classes.  

Now we come to the case of finite residue fields. It is easy to see that the algebraic de Rham complex $\Omega$ is continuous abelian sheaf which disposes of a transfer on the big \'etale (as well as Zariski) site of schemes, in other words it is a continuous object of the category $\Fan_{\text{\'et}}$. Thus by Corollary \ref{KtoFunctor}, the morphism above induces a morphism of cohomology groups for the improved Milnor $K$-sheaf
\begin{equation}\label{CompMilnorRigFin}\h^m(X,\widehat{\Ksheaf}^M_i)\rightarrow\HH^{i+m}\left(]\overline{X}[,j^\dagger\Omega_{]\overline{X}[}\right)=\h^{i+m}_{\text{rig}}(X/K).\end{equation}
Recall however that Rost's results state that the cohomology of a cycle module can be calculated by using the associated complex, which means in particular that we have to evaluate them solely on fields --- and on fields Kerz's usual and improved Milnor $K$-theories coincide. 

We have seen that the morphism (\ref{CompMilnorRig}) factors the rigid Chern classes of Petrequin in the case of a scheme with infinite residue fields. As a consequence of the remark in the previous paragraph and of the uniqueness property  in Corollary \ref{KtoFunctor} used for the construction of the Milnor Chern classes in the case of finite residue fields and also for the construction of morphism (\ref{CompMilnorRigFin}), we conclude that the morphism (\ref{CompMilnorRigFin}) factors the rigid Chern classes in the case of finite residue fields. \end{proof}

\subsection{The comparison theorem between rigid and overconvergent cohomology}\label{CompDLZSubsec}

From now on, let $X/k$ be smooth and quasiprojective and $K$ the fraction field of $W(k)$. There is a canonical comparison isomorphism 
$$\h^i_{\text{rig}}(X/K)\xrightarrow{\sim}\HH^i\left(W^\dagger\Omega_{X/k}\right)\otimes\QQ$$
between rigid and overconvergent de Rham cohomology constructed by Davis, Langer and Zink \cite{DavisLangerZink}.

They first construct it locally for affine schemes which posses a Witt lift using the universal property of K\"ahler differentials. It can be shown, that such a lift exists, and furthermore, that the comparison map is independent of the choice of Witt lift. The next step is to glue this construction wih the help of dagger spaces via a \v{C}ech spectral sequence. In this way they obtain a natural quasiisomorphism

\begin{equation}\label{CompDLZ}R\Gamma_{\text{rig}}(X)\xrightarrow{\sim} R\Gamma\left(X,W^\dagger\Omega_{X/k}\right)\otimes \QQ.\end{equation}


\begin{lem}\label{LemProd} The morphism (\ref{CompDLZ}) of Davis, Langer and Zink is compatible with the multiplicative structure on both sides: there is a commutative diagram
$$\xymatrix{R\Gamma_{\text{rig}}(X)\bigotimes^L R\Gamma_{\text{rig}}(X)\ar[r]\ar[d] & \left(R\Gamma(X,W^\dagger\Omega_{X/k})\otimes\QQ\right)\bigotimes^L\left(R\Gamma(X,W^\dagger\Omega_{X/k})\otimes\QQ\right)\ar[d]\\
R\Gamma_{\text{rig}}(X) \ar[r] & \left(R\Gamma(X,W^\dagger\Omega_{X/k})\otimes\QQ\right)}$$
where the tensor is taken over $W(k)$, the vertical maps represent the product and the upper horizontal map is given by the tensor product of the morphism (\ref{CompDLZ}).\end{lem}
\begin{proof}: Recall that $R\Gamma_{\text{rig}}(X)=R\Gamma(V,j^\dagger\Omega_V)$ where $V$ is a strict neighbourhood as above. Thus the left vertical map is given via the structure of $\Omega_V$ as differentially graded algebra. The same goes for the vertical map on the right hand side. Moreover, the comparison morphism (\ref{CompDLZ}) is by construction a morphism of differentially graded algebras, and consequently the diagram commutes.  \end{proof}

\subsection{Comparison with rigid Chern classes}

\begin{prop} Let $X$ be as above. The morphism (\ref{CompMilnorRig}) of cohomology groups $\h^i(X,\Ksheaf^M_m)\rightarrow \h^{i+m}_{\text{rig}}(X/K)$ factors through $\HH^{i+m}(X,W^\dagger\Omega)$.\end{prop}
\begin{proof}: Assume at first that $X$ has infinite residue fields. As $X$ is quasiprojective we can choose an open embedding $X\rightarrow \Proj\mathcal{S}$ where $\mathcal{S}$ is a finitely generated graded algebra over $k$. Davis, Langer and Zink point out that this enables us to consider finite coverings $\mathfrak{U}_X=\{U_i\}$ of $X$ fine enough such that the $U_i=\Spec A_i$ are standard smooth affines as well as their intersections. What is more, if we are given a global section $\overline{u}\in \OO_X^\ast$, we can choose the covering in a way that it represents a trivialisation for $\overline{u}$ given in local coordinates by
$$\overline{u}_{ij}\qquad\text{ on }U_i\cap U_j.$$
Further refinements allow us to consider $m$ such sections $\overline{u}^{(1)},\ldots,\overline{u}^{(m)}$ of $\OO_X^\ast$ each of which are given in local coordinates $\overline{u}_{ij}^{(k)}$. 

Thus we can consider a section of $\overline{\Ksheaf}^M_m$ locally given by
$$\{\overline{u}_{ij}^{(1)},\ldots,\overline{u}_{ij}^{(m)}\}\qquad\text{ on }U_i\cap U_j.$$
Its image in $W^\dagger\Omega\left[1\right]$ under $d\log$ is
$$\frac{d[\overline{u}_{ij}^{(1)}]}{[\overline{u}_{ij}^{(1)}]}\cdots\frac{d[\overline{u}_{ij}^{(m)}]}{[\overline{u}_{ij}^{(m)}]}\qquad\text{ on }U_i\cap U_j.$$
 
As in \cite{DavisLangerZink} we can choose for each $A_i$ a standard smooth lift $B_i$ over $W(k)$ with a fixed Frobenius lift such that for the lift $u^{(l)}_{ij}$ of each $\overline{u}^{(l)}_{ij}$  
$$F(u_{ij}^{(l)})=(u_{ij}^{(l)})^p,$$
and a homomorphism $\varkappa_i:B_i\rightarrow W(A_i)$, induced by $F$, which lifts $B_i\rightarrow A_i$, such that the image is overconvergent, thereby giving an overconvergent frame $(U_i,F_i,\varkappa_i)$ where $F_i=\Spec B_i$. By choosing the covering fine enough, we ensure that all intersections are again standard smooth affine and give rise to overconvergent frames. We denote the intersections by $U_I=\bigcap_{i\in I}U_i$ and $U_I=A_I$, where $I$ is a multi-index. 

Let as before $j:X\hookrightarrow \overline{X}$ be a suitable compactification and let $\overline{X}\hookrightarrow\yer$ be a closed immersion into a formal scheme over $\Spf(W(k))$. Denote by $\mathscr{U}_i=\Spf \mathscr{A}_i$ the formal completion of $F_i$ along $U_i$, which cover the image of $X$ in $\yer$. This can be completed to a covering $\mathfrak{U}$ of $\yer$ that induces the covering $\mathfrak{U}_X$. Further denote by $\mathfrak{U}_K$ the induced cover of the rigid generic fibre. Analogue to above we use the notations $\mathscr{U}_I$, $\mathscr{A}_I$ etc.

Again we choose liftings $\tilde{u}_{ij}^{(1)},\ldots,\tilde{u}_{ij}^{(m)}\in \mathscr{A}_{ij}$ of the local sections $\overline{u}_{ij}^{(1)},\ldots,\overline{u}_{ij}^{(m)}$. Then
$$\frac{d\tilde{u}_{ij}^{(1)}}{\tilde{u}_{ij}^{(1)}}\cdots\frac{d\tilde{u}_{ij}^{(m)}}{\tilde{u}_{ij}^{(m)}}$$
is a local section of $\Omega_{]\overline{X}[}$ and the image of $\{\overline{u}_{ij}^{(1)},\ldots,\overline{u}_{ij}^{(m)}\}\in\overline{\Ksheaf}^M_m$ under the map defined in Proposition \ref{RigChernFac}. The way these local sections were obtained implies that they glue to a global section $\frac{d\tilde{u}}{\tilde{u}}$ if we use dagger spaces. 

The task now is to check that the image of $\frac{d\tilde{u}_{ij}^{(1)}}{\tilde{u}_{ij}^{(1)}}\cdots\frac{d\tilde{u}_{ij}^{(m)}}{\tilde{u}_{ij}^{(m)}}$ under the comparison morphism of Davis,Langer and Zink is compatible with $\frac{d[\overline{u}_{ij}^{(1)}]}{[\overline{u}_{ij}^{(1)}]}\cdots\frac{d[\overline{u}_{ij}^{(m)}]}{[\overline{u}_{ij}^{(m)}]}$. 

Recalling that the $\overline{u}_{ij}^{(l)}\in A_{ij}$ and $\tilde{u}_{ij}^{(l)}\in B_{ij}$ are local coordinates and that we chose the Frobenius lift in a particular way, we see that the image of $\tilde{u}_{ij}^{(l)}$ under $\varkappa_{ij}$ is the Teichm\"uller lift $\left[\overline{u}_{ij}^{(l)}\right]$ (cf. \cite[Proposition 2.2.2]{Davis}). By the construction in \cite{DavisLangerZink} it follows that the map
$$\Gamma\left(\left]U_i\right[_{\mathscr{U}_i},\Omega_{\left]U_i\right[_{\mathscr{U}_i}}\right) \rightarrow  W\Omega_{U_i}\otimes\QQ$$
which is as a local map based upon the comparison map between the affine comparison morphism between Monsky--Washnitzer and overconvergent cohomology sends the class of $\frac{d \widetilde{u}_{ij}^{(l)}}{\widetilde{u}_{ij}^{(l)}}$ to the class of  $\frac{d[\overline{u}_{ij}^{(l)}]}{[\overline{u}_{ij}^{(l)}]}$ for all $i,j,l$. Although a priori this depends on the choice of Frobenius lift, Davis shows in \cite[Corollary 4.1.13]{Davis} that the comparison map is in fact independent of it. 

This morphism being the basis of the comparison morphism, we see that $\frac{d \widetilde{u}_{ij}^{(l)}}{\widetilde{u}_{ij}^{(l)}}$ is still mapped to $\frac{d[\overline{u}_{ij}^{(l)}]}{[\overline{u}_{ij}^{(l)}]}$ after passing to dagger spaces in order to glue. In particular, the \v{C}ech cocycle of rigid cohomology
$$\left(\frac{d \tilde{u}^{(l)}}{\tilde{u}^{(l)}}\right)\in Z^2(\mathfrak{U}_K,\Omega_{]X[}))$$
is sent to the \v{C}ech cocycle of overconvergent cohomology
$$\left(\frac{d[\overline{u}^{(l)}]}{[\overline{u}^{(l)}]}\right)\in Z^2(\mathfrak{U},W^\dagger\Omega_{X}),$$
for varying $l$. Thus, the same holds true for $\frac{d\tilde{u}^{(1)}}{\tilde{u}^{(1)}}\cdots\frac{d\tilde{u}^{(m)}}{\tilde{u}^{(m)}}$, which accordingly is sent to $\frac{d[\overline{u}^{(1)}]}{[\overline{u}^{(1)}]}\cdots\frac{d[\overline{u}^{(m)}]}{[\overline{u}^{(m)}]}$.

This shows that the induced morphism of cohomologies $\h^i(X,\overline{\Ksheaf}^M_m)\rightarrow \h^{i+m}_{\text{rig}}(X/K)$ factors indeed through $\HH^{i+m}(X,W^\dagger\Omega)$ and the diagram
$$\xymatrix{\h^i(X,\overline{\Ksheaf}^M_m)\ar[rr]\ar[dr] & & \h^{i+m}_{\text{rig}}(X/K)\ar[r]^{\sim\qquad} & \HH^{i+m}(X,W^\dagger\Omega)\otimes\QQ \\
& \HH^{i+m}(X,W^\dagger\Omega)\ar[urr]&&}$$
commutes.
 
In the case of finite residue fields, the statment follows as in Proposition \ref{RigChernFac} from the fact, that the morphisms $\h^i(X,\widehat{\Ksheaf}^M_m)\rightarrow \h^{i+m}_{\text{rig}}(X/K)$ and $\h^i(X,\widehat{\Ksheaf}^M_m)\rightarrow \HH^{i+m}(X,W^\dagger\Omega)$ are both deduced from the corresponding morphisms for Kerz's usual Milnor $K$-theory, and they are both unique by Corollary \ref{KtoFunctor}. \end{proof}
 
Now we can conclude with a comparison of rigid and overconvergent Chern classes.

\begin{theo}\label{CompChern} Let $X$ be a smooth quasiprojective scheme over $k$. The overconvergent Chern classes for $X$ defined here are compatible with the rigid Chern classes defined in \cite{Petrequin}.\end{theo}
\begin{proof}: We use the fact that the rigid and the overconvergent Chern classes factor through the Milnor $K$-sheaf. Consider the following diagram
$$\xymatrix{&&&\h^{2j-i}_{\text{rig}}(X/K)\\
K_j(X)\ar[rrru]^{c_{ij}^{\text{rig}}}\ar[rr]|{c_{ij}^M}\ar[rrrd]_{c_{ij}^{\text{sc}}}& & \h^{i-j}(X,\Ksheaf^M_i)\ar[ru]\ar[rd] & \\
&&&\HH^{2i-j}(X,W^\dagger\Omega)\ar[uu]}$$
where all the triangles commute: the upper left one by Proposition \ref{RigChernFac}, the lower left one by construction and the right one by the previous lemma. Given that all morphisms involved are compatible with products, this shows that the Chern classes in question are indeed compatible.
\end{proof}

\section{Overconvergent cycle classes}\label{Cycle}

In the light of the facts that, $\CH^i(X)=\h^i(X,\Ksheaf^M_i),$
and thus we have a morphism $\eta_{\text{sc}}^i: \CH^i(X)\rightarrow \HH^{2i}(X,W^\dagger\Omega_X^{\geqslant i}),$
and second, we can restrict the Chern classes $c_{ij}^{\text{sc}}:K_j(X)\rightarrow\h^{2i-j}(X,W^\dagger\Omega)$ to the $\gamma$-graded pieces whose sum is rationally isomorphic to Bloch's higher Chow groups $\CH^i(X,j)$, the goal of this section is to construct an integral morphism of higher cycle classes
$$\eta_{\text{sc}}^{ij}: \CH^i(X,j)\rightarrow \h^{2i-j}(X,W^\dagger\Omega^{\geqslant i}).$$

\subsection{Bloch's higher Chow groups}

Let $k$ be a field. Bloch makes in \cite{Bloch} the following definition of higher Chow groups which are in the case when $X/k$ is smooth, separated and $k$ perfect, equivalent to Voevodsky's motivic cohomology theory. 

Denote by $\Delta^N_k$ the standard algebraic $N$-simplex
$$\Delta^N_k:=\Spec k[t_0,\ldots,t_n]/(\sum t_i-1)$$
and by $\Delta_X^\ast$ the cosimplicial scheme given by
$$N\mapsto X\times_k\Delta^N_k.$$
The faces of $\Delta^N_X$ are defined by equations of the form $t_{i_1}=\cdots=t_{i_r}=0$. Let $z_r(X,i)$ be the subgroup of the cycles of dimension $r+i$ generated by the set of irreducible dimension $r+i$-subschemes of $\Delta_X^i$ that intersect all faces properly (i.e. that intersect each dimension $r$-face in dimension $\leqslant r+r$). Bloch's simplicial group is then given by
$$i\mapsto z_r(X,i).$$
The Chow groups (with respect to dimension) are then the homology groups of the associated complex
$$\CH_r(X,i)=\h_i(z_r(X,\ast)).$$
In the case when $X$ is equidimensional it is more convenient to label the complexes by codimension and define
$$\CH^r(X,i)=\h_i(z^r(X,\ast)),$$
where $z^r(X,i)=z_{n-r}(X,i)$ if $X$ is of dimension $n$. We may extend the definition of $z^r(X,i)$ to arbitrary smooth schemes by taking the direct sum over the irreducible components.

\subsection{Higher cycle classes with integral coefficients in the Milnor $K$-sheaf}

Let $X$ be smooth over a perfect field $k$ of characteristic $p>0$.

In \cite[\P 4]{Bloch} Bloch defines higher cycle classes from higher Chow groups into reasonable bigraded cohomology theories
$$\CH^b(X,n)\rightarrow\h^{2b-n}(X,b).$$
Namely, the cohomology theory has to be the hypercohomology of a complex which is contravariant. By replacing it with its Godement resolution it can be assumed to be built of acyclic sheaves. Moreover, Bloch assumes that for this theory one can define cohomology groups with supports, such that it satisfies a localisation sequence,  that it satisfies homotopy invariance, the existence of a cycle class for subschemes of pure codimension and weak purity. 

In order to construct higher cycle classes into the overconvergent cohomology groups, we take again the detour over Milnor $K$-theory using Rost's axiomatic. In particular we will not have to worry about the size of the residue fields of $X$.

\begin{lem} The cohomology groups of the Milnor $K$-sheaves satisfy the conditions required by Bloch in his construction of cycle classes.\end{lem}
\begin{proof}: We check the properties one by one. 
\begin{enumerate}\item \textbf{Calculated by a complex}. As we have seen in Corollary \ref{Cor65} and the subsequent remarks that according to Rost \cite{Rost} the sheaf cohomology of $\Ksheaf_b^M$ over $X$ can be calculated by the cohomology of the associated cycle complex $C^\ast(X;K^M_\ast,b)$
$$A^p(X;K^M_\ast,b)=\h^p(X,\Ksheaf^M_b).$$

\item \textbf{Localisation Sequence}. For a closed subscheme $i:Y\hookrightarrow X$ let $C_Y^\ast(X,K^M_\ast,b)$ be the ``cycle complex with supports'' defined by
$$C^p_Y(X;K^M_\ast,b)=\coprod_{\substack{x\in X^{(p)}\\x\in Y}}K^M_{b-p}(x),$$
and $A_Y^n(X;K^M_\ast,b)$ the associated cohomology with supports. 
Recall that there is a long exact sequence (\ref{LESfH}) for a closed subscheme $i:Y\hookrightarrow X$ and the associated immersion $j:U=X\backslash Y\rightarrow X$
$$\cdots\xrightarrow{\partial}A_p(Y;K^M_\ast,b)\xrightarrow{i_\ast}A_p(X;K^M_\ast,b)\xrightarrow{j_\ast}A_p(U;K_\ast^M,b)\xrightarrow{\partial}A_{p-1}(Y;K^M_\ast,b)\xrightarrow{i_\ast}\cdots.$$
This is in fact a localisation sequence: by definition we have Poincar\'e duality style equalities
\begin{eqnarray*} A_{n-p}(Y;K^M_\ast,n-b) &=& A^p_Y(X;K^M_\ast,b)\\
A_{n-p}(X;K^M_\ast,n-b) &=& A^p(X;K^M_\ast,b),\end{eqnarray*}
where $n$ is the relative dimension of $X$ over $k$. Thus the long exact sequence of homology above induces along exact sequence of cohomology
$$\cdots\xrightarrow{\partial}A^p_Y(X,K^M_\ast,b)\xrightarrow{i^\ast}A^p(X;K^M_\ast,b)\xrightarrow{j^\ast}A^p(U;K_\ast^M,b)\xrightarrow{\partial}A^{p+1}_Y(X;K^M_\ast,b)\rightarrow\cdots$$
and by the pointwise definition of cycle complexes, this sequence satisfies the usual functorial properties. 

\item \textbf{Homotopy invariance}. According to Rost (equation (\ref{HIC})), the cohomology groups $A^p(X;K^M_\ast,b)$ satisfy homotopy invariance
$$A^p(X;K^M_\ast,b)\cong A^p(X\times\aA^1;K^M_\ast,b).$$

\item \textbf{Cycle class}. Although this point does not hold for cycle modules in general, it holds for Milnor $K$-theory. Let $Y\subset X$ be of pure codimension $b$. Similarly to the discussion above there is an isomorphism
$$A^b_Y(X;K_\ast^M,b)\cong A^0(Y;K_\ast^M,0),$$
and the right hand sight is isomorphic to the zero cohomology group of $\Ksheaf^M_0$ on $Y$. But for any ring $K^M_0(A)=\ZZ$ and there is a well defined class 
$$\left[Y\right]\in A^b_Y(X;K_\ast^M,b)$$
that corresponds to the identity, which by construction is contravariant functorial with respect to the pull back of cycles.

\item \textbf{Weak purity}. Let $Y\subset X$ be of pure codimension $r$. There is an isomorphism
$$A^p_Y(X;K^M_\ast,b)\cong A_{n-p}(Y;K^M_\ast,b-n)),$$
where again $n$ is the dimension of $X$. The right-hand side is zero of $n-p$ is greater than the dimension of $Y$. Consequently the left-hand side is zero if $p<r$ .
\end{enumerate}\end{proof}

Now we can go step by step through the construction of cycle classes. Using that the cohomology of the Milnor $K$-sheaf can be calculated by a complex, we see that the usual diagram of simplices
\begin{equation}\label{chaincomplex}\xymatrix{X\ar@<.5ex>[r]\ar@<-.5ex>[r] & X\times\Delta^1 \ar@<1ex>[r]\ar[r]\ar@<-1ex>[r] & X\times\Delta^2\ar@<1.5ex>[r]\ar@<-1.5ex>[r]^{\vdots}&\cdots}\end{equation}
yields a double complex
$$C(X\times \Delta^\bullet;K^M_\ast,\bullet).$$
The first sheet of the spectral sequence associated to this double complex is given by
$$E_1^{pq}=A^q(X\times\Delta^{-p};K^M_\ast,b).$$
Note the appearance of a sign at the index $p$ on the right sight. This is due to the fact that the functor $A^p$ is contravariant in its first place and the introduction of a sign aligns the induced morphisms by the natural maps of (\ref{chaincomplex}) with the required structure of a spectral sequence. The sheet $E_1^{pq}$ is therefore by construction only non-zero for $p\leqslant 0$. In order to calculate the second sheet, we fix $q$ and look at the associated bounded complex
$$\cdots A^q(X\times\Delta^{-p};K^M_\ast,b)\xrightarrow{d_1^{pq}}A^q(X\times\Delta^{-(p+1)};K^M_\ast,b)\rightarrow\cdots\rightarrow A^q(X\times\Delta^1;K^M_\ast,b)\xrightarrow{d_1^{1q}}A^q(X;K^M_\ast,b)\rightarrow 0$$
where the boundary morphisms
$$d_1^{pq}:E_1^{pq}=A^q(X\times\Delta^{-p};K^M_\ast,b)\rightarrow E_1^{p+1,q}=A^q(X\times\Delta^{-p-1};K^M_\ast,b)$$
are induced by the pull-backs of the maps in (\ref{chaincomplex}). By the homotopy invariance of the cohomology $A^q$, 
$$A^q(X\times\Delta^{-p};K^M_\ast,b)\cong A^q(X;K^M_\ast,b)$$
for all $p\leqslant 0$. However as the simplexes collapse $\Delta^{-p}$ in the above complex, we discern from the definition of the boundary maps $d_1^{pq}$ that 
they are trivial if $p$ is odd and isomorphisms if $p$ is even. Therefore we find the second sheet to be
$$E_2^{pq}=\begin{cases}A^q(X;K^M_\ast,b) & \text{ for } p=0\\ 0 & \text{ otherwise.}\end{cases}$$
Hence the spectral sequence converges and we may write
\begin{equation}\label{SpectFirst}E_1^{pq}\Rightarrow A^\ast(X;K^M_\ast,b).\end{equation}
We get the same result if we truncate the diagram (\ref{chaincomplex}) at $X\times\Delta^N$ for $N$ even. 
The right-hand side of (\ref{SpectFirst}) is the target of our desired cycle map. We will use an auxiliary spectral sequence $\widetilde{E}^{pq}_r$ which maps into $E^{pq}_r$. 

Let $\widetilde{A}^a(X\times\Delta^p;K^M_\ast,b)=\varinjlim A^a_{|Z|}(X\times\Delta^p;K^M_\ast,b)$ where the limit is over $z^b(X,p)$ as used in the definition of the Chow groups and $|Z|$ denotes the support of $Z$. If we truncate again at some large even $N$ to avoid convergence problems, we get in the same manner as above another spectral sequence with the first sheet
$$\widetilde{E}_1^{pq}=\begin{cases} \widetilde{A}^q(X\times\Delta^{-p};K^M_\ast,b) & \text{ for } -p\leqslant N\\ 0 & \text{ otherwise.}\end{cases}$$
The natural morphism of cohomology groups from cohomology with supports to the regular one induces a map of spectral sequences
$$\widetilde{E}^{pq}_1\rightarrow E_1^{pq}.$$
Using Bloch's notation let $t_Nz^b(X,\cdot)$ be the truncation of the complex $z^b(X,\cdot)$ in degree $N$. Then the cycle class as described in the list above yields a morphism of complexes
\begin{equation}\label{Truncated}t_Nz^b(X,\cdot)\rightarrow \widetilde{E}_1^{\cdot,b}.\end{equation}
Note that as the limit in the definition of $\widetilde{A}^a(X\times\Delta^p;K^M_\ast,b)$ runs over cycles of pure codimension $b$ the weak purity axiom implies that  $\widetilde{E}_1^{pa}=\widetilde{A}^a(X\times\Delta^p;K^M_\ast,b)=0$ for $a<b$. Consequently this holds even for all sheets, i.e. $\widetilde{E}_r^{pa}=0$ for $a<b$. In particular, for $r>1$ this implies that the boundary maps
$$d_r^{pb}:\widetilde{E}^{p,b}_r\rightarrow\widetilde{E}_r^{p+r,b-r+1}=0$$
are zero as well. Taking cohomology on both sides of (\ref{Truncated}) and using the fact that the Chow groups derived from the untruncated complex $z^b(X,\cdot)$ maps into the truncated ones we get for any $n$
\begin{equation}\label{ChowFirst}\CH^b(X,n)\rightarrow \widetilde{E}_2^{-n,b}\rightarrow \widetilde{E}_{\infty}^{-n,b}.\end{equation}
Again by the weak purity axiom we see that $\widetilde{E}_{\infty}^{p,a}=0$ for $a<b$. Thus the morphism (\ref{ChowFirst}) maps in fact into the limit of the $\widetilde{E}_1$ spectral sequence in degree $b-n$. The morphism of spectral sequences $\widetilde{E}^{pq}_1\rightarrow E_1^{pq}$ induces then that (\ref{ChowFirst}) also maps into the limit of the $E_1$ spectral sequence in degree $b-n$ which is $A^{b-n}(X;K^M_\ast,b)$ as shown above. This concludes the construction and we get
\begin{cor} For a smooth scheme $X/k$ there is a family of cycle classes
\begin{equation}\label{CycleM}\eta^{bn}_M:\CH^b(X,n)\rightarrow A^{b-n}(X;K^M_\ast,b)=\h^{b-n}(X,\Ksheaf^M_b).\end{equation}
\end{cor}

We list some properties of the cycle class map for the Milnor $K$-sheaf.
\begin{description}\item[\textbf{Normalisation}.] The class of $X$ in the Chow ring $\CH^\ast(X,\ast)$ maps to the identity in the ring $\h^\ast(X,\Ksheaf^M_\ast)$. Indeed, we see that the cycle $\left[X\right]\in\CH^0(X,0)$ is mapped via the cycle map to the class of $X$ in $\h^0_X(X,\Ksheaf^M_0)=\h^0(X,\Ksheaf^M_0)$ which is isomorphic to $\ZZ$ and $\left[X\right]$ corresponds to the identity as we have seen above.

\item[\textbf{Functoriality with respect to flat pull-back and proper push-forward}.] Both the Chow ring and the cohomology of the Milnor $K$-sheaf are contravariant functorial with respect to flat pull-backs. Let $f:X'\rightarrow X$ be flat (a condition which can be dropped in case $X$ is smooth). Then Bloch shows in \cite[Prop.(1.3)]{Bloch2} that the complex that calculates the Chow groups is contravariant with respect to $f$, and consequently there is a well-defined pull-back map 
$$f^\ast:\CH^b(X,n)\rightarrow\CH^b(X',n).$$ 
Likewise Rost constructs in \cite[Section 12]{Rost} a pull-back map 
$$f^\ast:A^p(X;K^M_\ast,q)\rightarrow A^p(X';K^M_\ast,q)$$
coming from the corresponding pull-back map on the complex $C^p(X;K^M_\ast,q)$. 

The cycle class $\left[Y\right]\in\h^b(X,\Ksheaf^M_b)$ for subschemes $Y\subset X$ of pure codimension which plays an important role in the construction of the cycle class maps are contravariant functorial for morphisms $f:X'\rightarrow X$ which preserve the codimension. Thus, if we assume that $f$ is faithfully flat, we obtain functoriality of the cycle class maps $\eta_M^{bn}$ in the sense that the following diagram commutes
$$\xymatrix{\CH^b(X,n)\ar[r]^{\eta_M^{bn}}\ar[d]_{f^\ast} & \h^{b-n}(X,\Ksheaf^M_b))\ar[d]^{f^\ast}\\ \CH^b(X',n)\ar[r]^{\eta_M^{bn}} & \h^{b-n}(X',\Ksheaf^M_b)}$$

Even though we dispose in both cases of push-forwards for a proper morphism $f:X'\rightarrow X$, it is not clear to us yet, how to make use of it for the cycle class map, as Bloch points out that in case of the Chow groups the push-forward $f_\ast$ causes a shift in codimension by the degree of $f$ (\cite[Prop. (1.3)]{Bloch2}) which according to Rost \cite[3.5]{Rost} doesn't occur for his cycle complexes. 

\item[\textbf{Ring homomorphism}.] It is clear that the cycle class map is additive by linearity. In fact we have the following diagram
$$\xymatrix{\CH^b(X,n)\otimes\CH^b(X,n)\ar[r]^{\eta_M^{bn}\otimes\eta_M^{bn}} \ar[d] & \h^{b-n}(X,\Ksheaf_b^M)\otimes\h^{b-n}(X,\Ksheaf^M_b)\ar[d]\\
\CH^{b}(X\times X,n)\ar[r]^{\eta_M^{b,n}}\ar[d]_{\Delta^\ast} & \h^{b-n}(X\times X,\Ksheaf^M_{b})\ar[d]^{\Delta^\ast}\\
\CH^{b}(X,n)\ar[r]^{\eta_M^{b,n}} & \h^{b-n}(X,\Ksheaf^M_{b})}$$
where the upper square commutes by linearity and the lower one by pulling back along the diagonal. This extends of course linearly to addition of cycles of different codimension and degree. 

Multiplicativity requires more work. Multiplication in the higher Chow ring is described by Bloch in \cite[Section 5]{Bloch2}. In order to do this, it is sufficient to construct a map in the derived category for the corresponding complexes. More precisely, let $X$ and $Y$ be quasi-projective algebraic $k$-schemes. Then Bloch constructs a map
$$s\left(z^a(X,\cdot)\otimes z^b(Y,\cdot)\right)\rightarrow z^{a+b}(X\times Y,\cdot)$$
where on the left-hand side is $s$ denotes the simple complex associated with a double complex. The idea is to fix a triangulation for $\Delta^m\times\Delta^N\cong\aA^{m+n}$ for all $m$, $n$ such that it induces a well-defined morphism on the complexes. A triangulation is a family $T=\{T_{m,n}\}_{m,n\in\NN}$ with
$$T_{m,n}=\sgn(\theta)\theta$$
where $\theta$ is a face map $\Delta^{m+n}\rightarrow\Delta^m\times\Delta^n$. It is possible to fix a system of maps $T$ such that it induces a morphism of complexes $s\left(z^\ast(X,\cdot)\otimes z^\ast(Y,\cdot)\right)\rightarrow z^\ast(X\times Y,\cdot)$ where it is defined. However, the problem hereby is that $T_{n,m}\left(z^\ast(X,\cdot)\otimes z^\ast(Y,\cdot)\right)$ is not necessarily contained in $z^\ast(X\times Y,\cdot)$ as images of cycles might not meet all faces properly. The solution is to take the subcomplex of $s\left(z^\ast(X,\cdot)\otimes z^\ast(Y,\cdot)\right)$ generated by products $Z\otimes W$ such that $Z$ and $W$ are irreducible subvarieties of $X\times\Delta^m$ and $Y\times\Delta^n$ respectively and such that $Z\times W\subset X\times Y\times\Delta^m\times\Delta^n$ meets all faces properly. We denote this subcomplex by $z^\ast(X,Y,\cdot)'\subset s\left(z^\ast(X,\cdot)\otimes z^\ast(Y,\cdot)\right)$. Bloch shows in \cite[Theorem 5.1]{Bloch2} that this inclusion is in fact a quasi-isomorphism. As a consequence, one obtains a commutative diagram
$$\xymatrix{ s(z^{\ast}(X,\cdot)\otimes z^{\ast}(Y,\cdot))\ar[r]^{\qquad\sim}\ar[drr] & z^\ast(X,Y,\cdot)'\ar[r]^T & z^\ast(X\times Y,\cdot)\ar[d]\\
 & & z^\ast(X,\cdot) }$$
in the derived category and this induces an action of $\CH^\ast(Y,\cdot)$ on $\CH(X,\cdot)$. In particular, if $Y=X$ is smooth, one obtains a product on $\CH^\ast(X,\cdot)$ via pull-back along the diagonal
$$\CH^a(X,n)\otimes\CH^b(X,m)\rightarrow\CH^{a+b}(X\times X,n+m)\xrightarrow{\Delta^\ast}\CH^{a+b}(X,n+m)$$
which makes it into an anti-commutative ring \cite[Corollary 5.7]{Bloch}. 

By the above statements, it is sufficient, in order to see if the family of maps $\eta_M^{bn}$ is compatible with products, to consider the subcomplex $z^\ast(X,X,\cdot)'\subset s\left(z^\ast(X,\cdot)\otimes z^\ast(X,\cdot)\right)$. Thus let $Z\in z^a(X,n)$ and $W\in z^b(X,m)$ be irreducible subvarieties of $X\times\Delta^n$ and $X\times\Delta^m$ respectively such that $Z\times W\subset X\times X\times\Delta^n\times\Delta^m$ meets all faces of $\Delta^n\times\Delta^m$ properly, which means, that $Z\otimes W$ is in the set of generators of $z^\ast(X,X,\cdot)'$. The cycle class of Milnor $K$-theory mentioned above sends the class of  $Z$ to a unique class $[Z]\in A_Z^a(X\times \Delta^n,a)=A^0(Z;K^M_\ast,0)\cong\ZZ$ and $W$ to a unique class $[W]\in A_W^b(X\times\Delta^m,b)=A^0(W;K^M_\ast,0)\cong\ZZ$, which in both cases represents the identity. Rost's definition of (cross) products for cycle modules in \cite[Section 14]{Rost} 
$$C^p(Y;N,n)\times C^q(X,M,m)\rightarrow C^{p+q}(Y\times X;M)$$
holds in particular for the case of $N=M=K^M_\ast$. In this case the product is anti-commutative as shown in \cite[Corollary 14.3]{Rost}. Hence the product of $[Z]$ and $[W]$ as evoked above can easily be given as
\begin{eqnarray*}A^0(Z;K^M_\ast,0)\times A^0(W;K^M_\ast,0) &\rightarrow& A^0(Z\times W;K^M_\ast,0)\\
\left[Z\right]\times\left[W\right] &\mapsto&\left[Z\times W\right]\end{eqnarray*}
as all cycles involved represent the identity. Thus by means of the corresponding inclusions we obtain a commutative diagram
$$\xymatrix{z^a(X,n)\otimes z^b(X,m)\ar[r]\ar[d] & \widetilde{A}^a(X;K^M_\ast,n)\otimes\widetilde{A}^b(X;K^M_\ast,m)\ar[d]\\
z^{a+b}(X\times X, n+m)\ar[r] & \widetilde{A}^{a+b}(X\times;K^M_\ast,n+m)}$$
This shows that the morphism of complexes (\ref{Truncated}) is compatible with products and since this is the core of Bloch's construction of cycle class maps, they are compatible with products as well and one has a  diagram
$$\xymatrix{\CH^a(X,n)\otimes\CH^b(X,m)\ar[r]^{\eta_M^{an}\otimes\eta_M^{bm}} \ar[d] & \h^{a-n}(X,\Ksheaf_a^M)\otimes\h^{b-m}(X,\Ksheaf^M_b)\ar[d]\\
\CH^{a+b}(X\times X,n+m)\ar[r]^{\eta_M^{a+b,n+m}}\ar[d]_{\Delta^\ast} & \h^{a+b-n-m}(X\times X,\Ksheaf^M_{a+b})\ar[d]^{\Delta^\ast}\\
\CH^{a+b}(X,n+m)\ar[r]^{\eta_M^{a+b,n+m}} & \h^{a+b-n-m}(X,\Ksheaf^M_{a+b})}$$
where the upper square commutes due to the discussed reasons and the lower one again by pulling back along the diagonal.

\end{description}

\subsection{Higher cycle classes with integral coefficients in the overconvergent complex}

We now use the map (\ref{dlog}) of section \ref{TransMap}
$$d\log^n:\Ksheaf^M_n\rightarrow W^\dagger\Omega[n]$$
to define higher cycle classes with coefficients in the overconvergent cohomology theory. Remember that it induces a morphism of cohomology groups
$$\h^m(X,\Ksheaf_i^M)\rightarrow\HH^{m+i}(X,W^\dagger\Omega).$$
Note that whereas the first cohomology theory is bigraded, the second one is not. However, by definition the image of $d\log$ lies in the truncated complex $W^\dagger\Omega^{\geqslant n}[n]$ which is a subcomplex of $W^\dagger\Omega[n]$. As a consequence, $d\log$ factors and we can write
$$d\log^n:\Ksheaf^M_n\rightarrow W^\dagger\Omega[n],$$
and therefore on the cohomological level a morphism
$$\h^m(X,\Ksheaf^M_i)\rightarrow\HH^{m+i}(X,W^\dagger\Omega^{\geqslant i}).$$
Then the cycle class map for the Milnor $K$-sheaf (\ref{CycleM}) induces the following result.
\begin{prop} For $b,n\geqslant 0$ there exist cycle class maps
$$\eta_{\text{sc}}^{bn}:\CH^b(X,n)\rightarrow \HH^{2b-n}(X,W^\dagger\Omega^{\geqslant b}).$$\end{prop}

By functoriality of the morphism of cohomology rings $\h^\ast(X,\Ksheaf^M_\ast)\rightarrow\HH^\ast(X,W^\dagger\Omega^{\geqslant\ast})$ the cycle classes $\eta_{\text{sc}}^{bn}$ satisfy similar properties as mentioned above for the cycle classes $\eta_M^{bn}$.




\end{document}